\documentclass[12pt]{amsart}
\usepackage{a4wide}

\usepackage[T1]{fontenc}
\usepackage{epsfig}
\usepackage{mathtools}
\usepackage{amsmath, mathrsfs}
\usepackage{amssymb}
\usepackage{amsfonts}
\usepackage{soul}
\usepackage{graphicx}
\usepackage{xcolor}
\usepackage{amsthm}
\definecolor{deepblue}{RGB}{0,90,170}

\usepackage{amscd}
\usepackage{tikz}
\usepackage{amscd}
\usepackage{graphicx}
\usepackage{color}
\usepackage{verbatim}
\ProvideTextCommandDefault{\cprime}{\tprime}

\newtheorem{theorem}{Theorem}[section]
\newtheorem{lemma}[theorem]{Lemma}
\newtheorem{prop}[theorem]{Proposition}

\newtheorem{question}[theorem]{Question}

\newtheorem{remark}{Remark}

\renewcommand{\ge}{\geqslant}
\renewcommand{\geq}{\geqslant}
\renewcommand{\le}{\leqslant}
\renewcommand{\leq}{\leqslant}

\numberwithin{equation}{section}

\def \N {\mathbb N}

\def \Z {\mathbb Z}

\def \N {\mathbb N}

\def \Z {\mathbb Z}

\makeatletter
\@namedef{subjclassname@2020}{\textup{2020} Mathematics Subject Classification}
\makeatother

\makeindex

\begin{document}

	\title[Dimensional Entropy of Amenable Group Actions]{Dimensional Entropy of Amenable Group Actions\\ over Stable Sets and Fibres}

	\author{Xinyao He and Guohua Zhang}
	
	\address{\vskip 2pt \hskip -12pt Xinyao He}
	
	\address{\hskip -12pt School of Mathematical Sciences, Fudan University, Shanghai 200433, China}
	
	\email{hexy23@m.fudan.edu.cn}

	\address{\vskip 2pt \hskip -12pt Guohua Zhang}
	
	\address{\hskip -12pt School of Mathematical Sciences, Fudan University, Shanghai 200433, China}
	
	\email{chiaths.zhang@gmail.com}



	\parindent=10pt
	\begin{abstract}
This paper is devoted to the study of Bowen's dimensional entropy on subsets for actions of amenable groups. We prove three main results.
\begin{enumerate}

\item First, topological conditional entropy is characterized by the dimensional entropy of stable sets (Theorem~\ref{thm2}), answering a question of Dou, Wang and the second author of the present paper raised in [Fund. Math., 2025]. We remark that our Theorem~\ref{thm2} is the first characterization of topological conditional entropy via Bowen's dimensional entropy of stable sets even for $\mathbb{Z}$-actions.

\item  Second, we establish a dimensional entropy inequality for factor maps (Theorem~\ref{thm4}).
 It relates dimensional entropy of a set to that of its image and topological entropy of fibres, and may
be viewed as the dimensional-entropy counterpart of the factor-map inequality for packing
topological entropy due to Dou, Zheng, and Zhou proved as Theorem 1.4 in [Ergodic Theory Dynam. Systems, 2023].

\item  Third, the relative topological entropy of a factor map is determined by the dimensional entropy of the fibres (Theorem~\ref{thm3}). Notably, our proof of this formula  (Theorem~\ref{thm3}) is purely topological, in contrast to the recent measure-theoretic approach of Dou, Wang and Zhou based on relative Shannon--McMillan--Breiman theorems.
\end{enumerate}
 These results (Theorems~\ref{thm4} and \ref{thm3}) not only generalize the work of Oprocha and the second author of the present paper [Nonlinearity, 2011] from single transformations to amenable group actions, but also provide a purely topological and self-contained proof of a fibre entropy characterization recently obtained through measure-theoretic arguments.
	\end{abstract}
	
	\maketitle
	
	\setcounter{tocdepth}{1}
	
	
	\section{Introduction}

Entropy is one of the most fundamental tools for measuring the complexity of dynamical systems. Since its introduction by Kolmogorov in ergodic theory \cite{Kolmogorov1958} and by Adler, Konheim and McAndrew in topological dynamics \cite{Adler-Konheim-McAndrew1965}, entropy has played a central role in the study of dynamical systems
(see e.g.\ the survey by Katok \cite{Katok2007} or the book by Downarowicz \cite{Downarowicz2011}).
In 1973, Bowen \cite{R.Bowen1973-TAMS} introduced a dimensional entropy for subsets, defined in a manner resembling Hausdorff dimension, which has proved to be an important tool for describing the complexity carried by subsets (see e.g.\ the book by Pesin \cite{Pesin1997}).
Compared with topological entropy, dimensional entropy seems to give a more sensitive description of the complexity carried by subsets.
For instance, the dimensional entropy of any countable set is always zero, whereas the topological entropy of such a set may coincide with that of the whole space \cite{Ye-Zhang2007}; and for an invertible minimal system with finite topological entropy, there exists a dense family of subsets whose dimensional and topological entropies can be independently prescribed to any value $\alpha$ and $\beta$ (with $0\le \alpha\le \beta$) up to the system's total entropy \cite[Corollary 3.4]{Oprocha-Zhang2011}. The concept of packing topological entropy, as dynamical correspondence
of packing dimension, was introduced by Feng and Huang \cite{Feng-Huang2012}, where they established variational principles concerning dimensional entropy and packing topological entropy for analytic subsets in the system via the methods in geometric measure theory. Recently it was explored systematically in \cite{Wang-Zhang2025-preprint} the existence of measures of maximal dimensional entropy or packing topological entropy.

\smallskip

In recent years, much effort has been devoted to extending entropy theory beyond classical $\Z$-actions to more general groups, particularly amenable groups (see e.g.\ \cite{Ornstein-Weiss1987, Glasner-Thouvenot-Weiss2000, Lindenstrauss2001, Danilenko2001, Huang-Ye-Zhang2011, Dooley-Zhang2015, Downarowicz-Zhang2022-IJM, Downarowicz-Zhang2023, Dou-Zheng-Zhou2023, Dou-Wang-Zhang2025, Downarowicz-Weiss-Wiecek-Zhang2026-preprint, Dou-Wang-Zhou2026-preprint}).

 Recall that a countable discrete group $G$ is \emph{amenable} if it admits a sequence $\{F_n\}_{n\in\N}\subset\mathfrak{F}_G$ of finite subsets (called a \emph{F\o lner sequence}) such that
\[
\lim_{n\to\infty}\frac{|F_n\cap gF_n|}{|F_n|}=1\quad\text{for each }g\in G,
\]
where $|\cdot|$ denotes the cardinality and $\mathfrak{F}_G$ is the family of all non-empty finite subsets of $G$. A F\o lner sequence $\{F_n\}_{n\in\N}$ is \emph{tempered} if there exists a constant $C>1$ such that
\begin{equation}\label{eq:shulman}
\Bigl|\bigcup_{k<n}F_k^{-1}F_n\Bigr|\le C|F_n|\quad\text{for all }n\ge 2.
\end{equation}
	It is well known that
the cardinality of every F\o lner sequence in each countably infinite discrete amenable group will tend to infinity, and that
 every countably infinite discrete amenable group admits a tempered F\o lner sequence $\{F_{n}\}_{n\in\mathbb{N}}\subset\mathfrak{F}_G$ satisfying
	\begin{equation}	\label{1.1}
		\lim_{n\to\infty}\dfrac{|F_{n}|}{\log n}=\infty.
	\end{equation}

Throughout this paper, $G$ denotes a countably infinite discrete amenable group along with the unity $e_G$ and a F\o lner sequence $\{F_{n}\}_{n\in\mathbb{N}}\subset\mathfrak{F}_G$, and $(X,G)$ denotes a $G$-action: a non-empty compact metric space $X$ with a compatible metric $d$, together with a group $G$ of self-homeomorphisms of $X$ (with $e_G$ denoting the identity transformation).

\smallskip

In this paper, we study Bowen's dimensional entropy for amenable group actions from two perspectives: its
relation to topological conditional entropy via stable sets, and its behaviour under factor maps. Recall that $\pi: (X, G)\to (Y, G)$ is
 a \emph{factor map} between two $G$-actions, if $\pi: X\to Y$ is a continuous surjection
with $\pi\circ g=g\circ\pi$ for all $g\in G$.

\smallskip

Let us overview our main results in the order proved in the paper. Note that Theorems~\ref{thm4} and \ref{thm3} generalize main results of \cite{Oprocha-Zhang2011} from $\mathbb{Z}$-actions to amenable group actions.
\begin{enumerate}

\item Following \cite{R.Bowen1972, Misiurewicz1976-SM}, the topological conditional entropy of amenable group actions was characterized recently by topological entropy of stable sets \cite[Theorem 2.2]{Dou-Wang-Zhang2025}. Our first result Theorem \ref{thm2} answers
    \cite[Question 2.8]{Dou-Wang-Zhang2025}, which provides a characterization of the topological conditional entropy of amenable group actions by the dimensional entropy of stable sets. We remark that our Theorem~\ref{thm2} is the first characterization of topological conditional entropy via Bowen's dimensional entropy of stable sets even for $\mathbb{Z}$-actions.

\item In \cite{Dou-Zheng-Zhou2023} Dou, Zheng and Zhou
gave a systematic study of packing topological entropy for amenable group
actions, establishing a variational principle and entropy inequalities for factor
maps. Our second result Theorem \ref{thm4} concerns an upper bound for the dimensional entropy of a set under a
factor map between amenable group actions, which can be viewed as
the counterpart of \cite[Theorem 1.4]{Dou-Zheng-Zhou2023} within the dimensional entropy framework, complementing their study from the dimensional entropy side.

\item
Our third result Theorem \ref{thm3} provides a characterization of the relative topological entropy in terms of the dimensional entropy of the fibres for a factor map between amenable group actions. While building a preprint version of the paper,
we learned from Dou that Theorem \ref{thm3} has also been
obtained independently in their recent preprint as \cite[Theorem 1.3]{Dou-Wang-Zhou2026-preprint}, by a measure-theoretic approach.
 However, Theorem \ref{thm3} was proved in this paper by purely topological arguments.
\end{enumerate}

\subsection*{Statements of the main results and strategy}

Our three main theorems are proved by a unified methodology
based on the Lindenstrauss covering lemma (Lemma \ref{coveringlem}): each theorem relies on a key
covering proposition (Propositions \ref{prop14}, \ref{section-3}, and \ref{prop15}, respectively) that provides quantitative
control over covering numbers in terms of dimensional entropy. The covering lemma
then allows us to combine these local estimates into global entropy bounds. More precisely, the main theorems of
this paper all rely on Lemma \ref{coveringlem} as their common
combinatorial backbone. Each proof employs the lemma in a slightly different manner: for
Theorem \ref{thm2}, it converts local stable-set entropy bounds into global covering estimates; for
Theorem \ref{thm4}, it combines fibrewise and base-space coverings into a single estimate; for Theorem \ref{thm3},
 it yields a uniform covering bound for all fibres simultaneously.

\smallskip

For $\varepsilon > 0$ and $x\in X$, the \emph{$\varepsilon$-stable set} of $x$ is defined by
\[\Gamma_\varepsilon(x)=\{y\in X: d(gx,gy)\le\varepsilon,\ \forall g\in G\}.\]
In \cite{Dou-Wang-Zhang2025} the authors extended Bowen's classical result \cite{R.Bowen1972} from $\mathbb{Z}$-actions to amenable groups, by showing that the topological
conditional entropy $h^*(X,G)$ can be characterized by the topological entropy of local stable
sets \cite[Theorem 2.2]{Dou-Wang-Zhang2025}:
\[
h^{*}(X,G)=\lim_{\varepsilon\to 0}\sup_{x\in X}
	h_{\mathrm{top}}(\Gamma_{\varepsilon}(x),\{F_{n}\}_{n\in\mathbb{N}}).
	\]
 They further raised \cite[Question 2.8]{Dou-Wang-Zhang2025}: whether such a
characterization remains valid when topological entropy is replaced by Bowen's dimensional
entropy $h_{\dim}$.

 The following theorem not only answers this question affirmatively for a tempered F\o lner sequence satisfying (\ref{1.1}), but also provides the first characterization of topological conditional entropy via Bowen's dimensional entropy of stable sets even for $\mathbb{Z}$-actions.

	\begin{theorem}\label{thm2}
		Let $\{F_{n}\}_{n\in\mathbb{N}}$ be a tempered F\o lner sequence satisfying (\ref{1.1}). Then
		$$
		h^{*}(X,G)=\lim_{\varepsilon\to 0}\sup_{x\in X}
		h_{\mathrm{dim}}(\Gamma_{\varepsilon}(x),\{F_{n}\}_{n\in\mathbb{N}}).
		$$
	\end{theorem}
	
	\begin{remark}
		It was also considered in \cite{Dou-Wang-Zhang2025} the following version of \(\varepsilon\)-stable set of \(x\):
		\begin{equation} \label{stable}
			\Gamma_{\varepsilon}^{\left\{F_n\right\}_{n \in \mathbb{N}}}(x)=\bigcap_{n \in \mathbb{N}} \{ y \in X: d(gx, gy) \le \varepsilon,\ \forall g \in F_n \} \supset \Gamma_{\varepsilon}(x).
		\end{equation}
		From \cite[Theorem 2.2]{Dou-Wang-Zhang2025} and the definition of the Bowen's dimensional entropy presented in the next section, Theorem~\ref{thm2} remains true if $\Gamma_\varepsilon(x)$ is replaced by
		$\Gamma_{\varepsilon}^{\{F_n\}_{n \in \mathbb{N}}}(x)$:
\begin{equation} \label{more}
h^{*}(X,G)= \lim_{\varepsilon\to 0}\sup_{x\in X}
			h_{\mathrm{dim}}(\Gamma_{\varepsilon}^{\{F_{n}\}_{n\in\mathbb{N}}}(x),\{F_{n}\}_{n\in\mathbb{N}}).
\end{equation}
In fact, as argued in the proof of Theorem \ref{thm2}, one already obtains
		\begin{equation*}
			h^{*}(X,G)\geq \lim_{\varepsilon\to 0}\sup_{x\in X}
			h_{\mathrm{dim}}(\Gamma_{\varepsilon}^{\{F_{n}\}_{n\in\mathbb{N}}}(x),\{F_{n}\}_{n\in\mathbb{N}})
			\geq \lim_{\varepsilon\to 0}\sup_{x\in X}
			h_{\mathrm{dim}}(\Gamma_{\varepsilon}(x),\{F_{n}\}_{n\in\mathbb{N}})
		\end{equation*}
by \cite[Theorem~2.2 and Proposition~5.1]{Dou-Wang-Zhang2025}, and then the identity \eqref{more} by Theorem \ref{thm2}.
	\end{remark}

The following two
results generalize the work of Oprocha and the second author of the present paper \cite{Oprocha-Zhang2011} from
single transformations to amenable group actions. We remark that, for $\mathbb{Z}$-actions, the corresponding results in \cite[Theorems 4.2 and 5.1]{Oprocha-Zhang2011} follow directly from Theorems~\ref{thm4} and \ref{thm3} by taking $F_n=\{0, 1, \cdots,n-1\}$ for each $n\in\mathbb N$.

\smallskip

Our second result concerns an upper bound for the dimensional entropy of a set under a factor
map. We also remark that the following inequality holds along any F\o lner
sequence: neither temperedness nor the growth condition \eqref{1.1} is required.

	\begin{theorem}\label{thm4}
Let $\pi: (X, G)\to (Y, G)$ be a factor map between two $G$-actions.
		Then
		$$
		h_{\mathrm{dim}}(E,\{F_{n}\}_{n\in\mathbb{N}})
		\le
		h_{\mathrm{dim}}(\pi(E),\{F_{n}\}_{n\in\mathbb{N}})
		+\sup_{y\in Y}h_{\mathrm{top}}(\pi^{-1}(y),\{F_{n}\}_{n\in\mathbb{N}}),\ \ \ \forall \emptyset\neq E\subset X.
		$$
	\end{theorem}

Note that, the second item on the
right-hand side of Theorem \ref{thm4} is exactly $h_{\mathrm{top}}(G,X|\pi)$, the relative topological entropy of the factor map $\pi$, that is,
	$$
	h_{\mathrm{top}}(G,X|\pi)=\sup_{y\in Y}h_{\mathrm{top}}(\pi^{-1}(y),\{F_n\}_{n\in\mathbb N}),
	$$
which was first proved by Dooley and the second author of the paper in \cite[Theorem 13.3]{Dooley-Zhang2015} and later by Yan as \cite[Theorem 5.1]{Yan2015}. Our third result provides another characterization of the relative topological entropy of a factor map in terms of the
dimensional entropy of the fibres, although in general topological entropy and dimensional entropy of subsets may differ (see for example \cite{Oprocha-Zhang2011})
as explained at the beginning of the paper.

	\begin{theorem}\label{thm3}
Let $\pi: (X, G)\to (Y, G)$ be a factor map between two $G$-actions, and let $\{F_{n}\}_{n\in\mathbb{N}}$ be a tempered F\o lner sequence satisfying (\ref{1.1}). Then
		$$
		h_{\mathrm{top}}(G,X|\pi)
		=\sup_{y\in Y}h_{\mathrm{dim}}(\pi^{-1}(y),\{F_{n}\}_{n\in\mathbb{N}}).
		$$
	\end{theorem}

\begin{remark}
When $E=X$, Theorems~\ref{thm4} and \ref{thm3} yield directly the inequality
$$
h_{\mathrm{top}}(X, G) \leq \sup _{y \in Y} h_{\operatorname{dim}}\left(\pi^{-1}(y),\left\{F_n\right\}_{n \in \mathbb{N}}\right)+h_{\mathrm{top}}(Y, G),
$$
which is precisely the inequality \cite[Corollary 1.4]{Dou-Wang-Zhou2026-preprint} via their measure-theoretic approach. However, our Theorem \ref{thm4} is more general as it applies to arbitrary subsets $E \subset X$, not only to the whole space $E=X$. Moreover, our proof is purely topological and does not rely on the Shannon-McMillan-Breiman theorem or the variational principle.
\end{remark}

\subsection*{Further comments}

While preparing this manuscript, we learned from Dou that Theorem
\ref{thm3} was established as \cite[Theorem 1.3]{Dou-Wang-Zhou2026-preprint} via a completely different route. Their proof proceeds as follows: they prove firstly a relative Shannon--McMillan--Breiman theorem for conditional entropy with
respect to a $G$-invariant sub-$\sigma$-algebra, then deduce a relative Brin--Katok entropy formula and finally obtain the
required conclusion by applying the Brin--Katok formula to the special case of fibre measures along with the variational principle of conditional entropy.

However, in this paper all results are proved by purely topological arguments, which not only yield new
insights into the combinatorial nature of dimensional entropy, but also offer some
methodological advantages over the measure-theoretic approach in \cite{Dou-Wang-Zhou2026-preprint}.
On one hand, the measure-theoretic proofs in \cite{Dou-Wang-Zhou2026-preprint} rely crucially on the relative
Shannon--McMillan--Breiman theorem (which was established till now only for amenable group actions) and the variational principle for conditional entropy, in contrast, our topological approach works directly with open covers and Bowen balls,
revealing that the entropy formulae in question are in fact consequences of a more elementary
combinatorial structure, namely the covering properties encoded in the Lindenstrauss covering
lemma (Lemma~\ref{coveringlem}). On the other hand, by bypassing the variational principle entirely, our approach demonstrates that the
coincidence of relative topological entropy with the supremum of fibre dimensional entropies is
a structural property of the topological dynamics itself, not merely a consequence of the equality
between measure-theoretic and topological entropies. This structural viewpoint may open
avenues for extending such results beyond the amenable setting, where variational principles are
not yet fully established.

Nevertheless, both approaches have their own merits. The measure-theoretic approach presented in \cite{Dou-Wang-Zhou2026-preprint} places the fibre
entropy formula within a broader framework of conditional entropy theory, connecting it to some deep results in ergodic theory.
Our topological approach reveals that the formula is in fact a consequence of elementary combinatorial properties of amenable group
actions, and may therefore be more readily adaptable to settings where the measure-theoretic machinery is not
available. The coexistence of these independent proofs may strengthen our understanding of the
result and suggests that it reflects a fundamental structural feature of dynamical systems over factors.

\subsection*{Two remaining problems}

We end this section with two related problems.

\smallskip

It seems that the assumption of a F\o lner sequence being tempered and
satisfying (\ref{1.1}) is very natural in Theorem \ref{thm2}, note that \cite[Theorem 2.4]{Dou-Wang-Zhang2025}
provides another characterization of topological
conditional entropy using Bowen's idea of dimensional entropy under the same assumption.
We also note that if we let $F_n=\{0, 1, \cdots, n\}$ for each $n \in \mathbb{N}$, then it forms a F\o lner sequence of $\mathbb{Z}$ with all the required properties.

However, note that \cite[Theorem 2.2]{Dou-Wang-Zhang2025} was proved there for any F\o lner sequence, and that it was raised as
\cite[Question 2.9]{Dou-Wang-Zhang2025} if \cite[Theorem 2.4]{Dou-Wang-Zhang2025} holds for arbitrary F\o lner sequence. Thus it
seems reasonable to ask whether either of these assumptions of temperedness and
the growth condition (\ref{1.1}) can be removed or weakened.
We remark that, in the proof of Theorem \ref{thm2}, the temperedness assumption and
the growth condition (\ref{1.1}) are only used when we apply Proposition \ref{prop14}. Thus the answer to the question below will be
affirmative once Proposition \ref{prop14}
can be proved for any F\o lner sequence.

\begin{question}
Does Theorem \ref{thm2} hold along arbitrary F\o lner sequences?
\end{question}

Our topological approach relies crucially on the Lindenstrauss covering lemma (Lemma \ref{coveringlem}), which is specific to amenable groups. For sofic
group actions, where a version of entropy theory has been developed recently by many researchers,
the appropriate covering lemmas are not yet available. Thus the following question arises naturally:

\begin{question}
Is there a generalization of our combinatorial method to the sofic setting? Alternatively, is there a purely topological proof of
the fibre entropy formula for sofic groups?
\end{question}

\subsection*{Organization of the paper} The remainder of the paper is organized as follows. Section~\ref{preli} collects some preliminaries on amenable groups and several notions of entropy. Section~\ref{conditional} proves Theorem~\ref{thm2}, giving a characterization of topological conditional entropy via dimensional entropy of stable sets.
The proof relies on Proposition \ref{prop14}, which provides a uniform covering estimate controlling the topological
conditional entropy via the dimensional entropy of stable sets.
Section~\ref{inequality for dim} establishes a dimensional entropy inequality for factor maps (cf. Theorem~\ref{thm4}),
with Proposition \ref{section-3} as the
key covering tool that combines fibrewise estimates with a dimensional entropy bound
on the image set. Section~\ref{relative} treats dimensional entropy in the fibred setting and proves Theorem~\ref{thm3}, relying on Proposition
\ref{prop15} for a uniform covering estimate of fibres.

	\section{Preliminaries} \label{preli}
	In this section let us recall some necessary preliminaries.
	
	\subsection{Basic facts on amenable groups and their actions}\
	
	We begin with some basic ingredients from the entropy theory of amenable group actions. One of the fundamental tools in this theory is the Ornstein-Weiss convergence lemma proved firstly in \cite{Ornstein-Weiss1987}. We use the formulation given in \cite[1.3.1]{Gromov1999}.

	\begin{lemma} \label{OW-Lemma}
		Let $f: \mathfrak{F}_G \rightarrow \mathbb{R}_+$ be an invariant subadditive function, that is,
		$f (F g)= f (F)$ and $f (E\cup F)\le f (E)+ f (F)$ whenever $E, F\in \mathfrak{F}_G$ and $g\in G$.
		Then the limit $\lim\limits_{n\to \infty} \frac{f (F_n)}{|F_n|}$ exists and is independent of the choice of the F\o lner sequence $\{F_{n}\}_{n\in\mathbb{N}}$.
	\end{lemma}

We shall use the following easy observation, whose proof is included for completeness.
	
	\begin{lemma}\label{prop-folnerseq}
		Fix $A\in\mathfrak{F}_G$  and let $F_n^A=F_n\cup A$ for each $n\in\mathbb N$. Then $\{F_n^A\}_{n\in\mathbb N}$ is again a F\o lner sequence. Moreover, if $\{F_n\}_{n\in\mathbb N}$ is tempered, then so is $\{F_n^A\}_{n\in\mathbb N}$.
	\end{lemma}
	
	\begin{proof}
The easy fact that $\{F_n^A\}_{n\in\mathbb N}$ is a F\o lner sequence comes from the estimation below:
		$$
		\liminf_{n\to\infty}\frac{|gF_n^A\cap F_n^A|}{|F_n^A|}
		\geq
		\liminf_{n\to\infty}\frac{|gF_n\cap F_n|}{|F_n|+|A|}
		=1,\ \ \  \forall g\in G.
		$$		
Now if $\{F_n\}_{n\in\mathbb N}$ is a tempered F\o lner sequence with constant $C> 1$, then for each $i\geq 2$:
		$$
		\begin{aligned}
			\bigg|\bigcup_{j<i}(F_j^A)^{-1}F_i^A\bigg|
			&\leq
			\bigg|\bigcup_{j<i}F_j^{-1}F_i\bigg|
			+
			\bigg|\bigcup_{j<i}F_j^{-1}\bigg|\cdot |A|
			+
			|A|\cdot |F_i|
			+
			|A|^2 \\
			&\leq
			C|F_i|+C|F_i|\cdot |A|+|A|\cdot |F_i|+|A|^2 \\
			&\leq
			(C+|A|)(1+|A|)\,|F_i^A|.
		\end{aligned}
		$$
In particular, $\{F_n^A\}_{n\in\mathbb N}$ is also tempered, but with constant $(C+|A|)(1+|A|)$.
	\end{proof}

	The following combinatorial covering lemma due to Lindenstrauss \cite[Corollary 2.7]{Lindenstrauss2001} is central to our proofs.
	Recall that a family \(\mathcal{F} \subset \mathfrak{F}_G\) is \emph{\(\gamma\)-disjoint} (\(0 \le \gamma < 1\)) if for each \(A \in \mathcal{F}\) there exists \(A' \subset A\) with \(|A'| \ge (1-\gamma)|A|\) such that the sets \(A'\) are pairwise disjoint for different \(A \in \mathcal{F}\). In
	particular, a 0-disjoint family means equivalently a disjoint family.
	
	\begin{lemma}\label{coveringlem}
		For any \(\delta \in (0, 10^{-4})\), \(C > 0\), and \(D \in \mathfrak{F}_G\), there exists \(M \in \mathbb{N}\) large enough (depending on \(\delta, C, D\)) such that: Once \(H \in \mathfrak{F}_G\) and families \(\{H_{i,j}: i=1,\cdots,M, j=1,\cdots,N_i\}\) and \(\{A_{i,j}: i=1,\cdots,M, j=1,\cdots,N_i\}\) satisfy:
		\begin{enumerate}

			\item
			Letting \(H_{i,*} = \bigcup_{j=1}^{N_i} H_{i,j}\), we have for all \(i=1,\cdots,M, k=2,\cdots,N_i\),
			\[
			\left| \bigcup_{k'<k} H_{i,k'}^{-1} H_{i,k} \right| \le C |H_{i,k}|,
			\]
			and for all \(i=2,\cdots,M, k=1,\cdots,N_i\),
			\[
			\left| \bigcup_{i'<i} D H_{i',*}^{-1} H_{i,k} \right| \le (1+\delta) |H_{i,k}|.
			\]
			
			\item For all \(i=1,\cdots,M, j=1,\cdots,N_i\), we have \(H_{i,j} A_{i,j} \subset H\).
		\end{enumerate}
		Define \(A_{i,*} = \bigcup_{j=1}^{N_i} A_{i,j}\) and \(\alpha = \frac{\min_{1\le i \le M} |D A_{i,*}|}{|H|}\). Then the collection \(\{H_{i,j}a: a \in A_{i,j}, i=1,\cdots,M, j=1,\cdots,N_i\}\) admits a \(10\delta^{\frac{1}{4}}\)-disjoint subfamily \(\mathcal{F}\) satisfying \(|\cup \mathcal{F}| \ge (\alpha - \delta^{\frac{1}{4}}) |H|\).
	\end{lemma}

\begin{remark}
Our three main results use above Lemma \ref{coveringlem} via the same general pattern: construct auxiliary sets $H_{i,j}$ and
 $A_{i,j}$ that encode the dynamical information at different scales, then verify the temperedness-like conditions of the lemma by choosing the sets
 $H_{i,j}$ appropriately, and finally extract the disjoint subfamily $\mathcal{F}$ and use its covering properties to obtain entropy
estimates. The differences between these proofs lie in how the sets $H_{i,j}$ are constructed and what dynamical information they encode.
\end{remark}
	
	We also employ the following result from \cite[Lemma 5.3]{Dou-Wang-Zhang2025}, which is a variant of a combinatorial lemma due to Lindenstrauss \cite[Lemma 4.2]{Lindenstrauss2001}.

	\begin{lemma}\label{combilem}
		Assume that the F\o lner sequence $\{F_{n}\}_{n\in\mathbb{N}}$ satisfies \eqref{1.1}. Then, for any $\eta>0$, there exists $N=N(\eta)\in \mathbb{N}$ such that for all indices $N\leq n_{1}<\cdots<n_{r}$, the number of $\frac{1}{2}$-disjoint sub-collections of $\{F_{n_{i}}a:F_{n_{i}}a\subset F_{n},i=1,\cdots,r\}$ is at most $2^{\eta|F_{n}|}$ for every sufficiently large $n\in \mathbb{N}$ (depending on $n_{1},\cdots,n_{r}$).
	\end{lemma}

	\subsection{Various entropies of amenable group actions}\
	
	In this subsection, we recall various entropies of amenable group actions used later.
	
	Denote by $\mathfrak{C}_{X}^{o}$ the set of all finite open covers of $X$. For $\mathcal{U}\in \mathfrak{C}_{X}^{o}$ and $\emptyset\neq K\subset X$, set $N(\mathcal{U},K)$ to be the minimal cardinality of sub-families of $\mathcal{U}$ covering $K$. For $\mathcal{U}, \mathcal{V}\in \mathfrak{C}_{X}^{o}$, write $\mathcal{V}\succeq\mathcal{U}$ if each element of $\mathcal{V}$ is contained in some element of $\mathcal{U}$.
 Set $\mathcal{U}\vee \mathcal{V}=\{U\cap V:\ U\in \mathcal{U},\ V\in \mathcal{V}\}$, which in fact works similarly for any finite collection of covers. Given $F\in \mathfrak{F}_{G}$, we use the notation $\mathcal{U}^{F}=\bigvee_{g\in F} g^{-1}\mathcal{U}$, where $g^{-1}\mathcal{U}=\{g^{-1}U:\ U\in \mathcal{U}\}$; and define
 $$
B_\eta^F(x)=\{y \in X: d(g x, g y) \leqslant \eta \text { for all } g \in F\}\ \ \  \text {for $x\in X$ and $\eta> 0$,}
$$
which is a closed ball with center $x$ and radius $\eta$ under the metric $d_F$ where
\[d_F(y, z)=\max_{g \in F} d(g y, g z)\ \ \ \ \ \text{for $y, z\in X$}.\]

	\subsubsection{Topological conditional entropy}
	
The notion of topological conditional entropy was first introduced by Misiurewicz \cite{Misiurewicz1976-SM}, and then
 in the more general setting of sofic groups by Chung and the second author of the present paper \cite{Chung-Zhang2015} (see also \cite{Zhou-Zhang-Chen2015} for the setting of amenable groups). Here we restrict ourselves to amenable group actions.

For $\mathcal{U}, \mathcal{V}\in \mathfrak{C}_{X}^{o}$, let $N(\mathcal{U}|\mathcal{V})=\max\{N(\mathcal{U},V):V\in \mathcal{V}\}$. It is easy to show that the function $\mathfrak{F}_{G}\ni F\mapsto \log N(\mathcal{U}^{F}|\mathcal{V}^{F})$ is invariant and subadditive \cite[Lemma 5.2]{Chung-Zhang2015}, so we could apply Ornstein-Weiss convergence lemma (Lemma \ref{OW-Lemma}) to define
\[ \begin{aligned} h(G,\mathcal{U}|\mathcal{V})=\lim_{n\to \infty}\frac{1}{|F_{n}|}\log N(\mathcal{U}^{F_{n}}|\mathcal{V}^{F_{n}}), \end{aligned} \] whose value is independent of the choice of the F\o lner sequence $\{F_{n}\}_{n\in \mathbb{N}}$.
	The \emph{topological conditional entropy of the $G$-action $(X,G)$} is defined as
\[
h^{*}(X,G)=\inf_{\mathcal{V}\in \mathfrak{C}_{X}^{o}}\sup_{\mathcal{U}\in \mathfrak{C}_{X}^{o}}h(G,\mathcal{U}|\mathcal{V}).
\]
	
	\subsubsection{Relative topological entropy}

Let $\pi: (X, G)\to (Y, G)$ be
a factor map between $G$-actions. For $\mathcal{U}\in \mathfrak{C}_{X}^{o}$, let
$N(\mathcal{U}|\pi)=\sup_{y\in Y}N(\mathcal{U},\pi^{-1}(y))$.  Again, the function $\mathfrak{F}_{G}\ni F\mapsto \log N(\mathcal{U}^{F}|\pi)$ is invariant and subadditive, so we
may apply Lemma \ref{OW-Lemma} to put
	\[
	h_{\mathrm{top}}(G,\mathcal{U}|\pi)=\lim_{n\to \infty}\frac{1}{|F_{n}|}\log N(\mathcal{U}^{F_{n}}|\pi),
	\]
and then we define the \emph{relative topological entropy relative to $\pi$} as
	\[
	h_{\mathrm{top}}(G,X|\pi)=\sup_{\mathcal{U}\in \mathfrak{C}_{X}^{o}}h_{\mathrm{top}}(G,\mathcal{U}|\pi).
	\]
When the $G$-action $(Y, G)$ is trivial, that is, $Y$ is a singleton, then $h_{\mathrm{top}}(G,X|\pi)$ is also denoted by $h_{\mathrm{top}}(G,X)$, and we call it the \emph{topological entropy of $G$-action $(X, G)$}.

	\subsubsection{Topological and dimensional entropies of subsets}
	
For  \(\emptyset\neq E\subset X\) and $\mathcal{U}\in\mathfrak{C}_{X}^{o}$, let
	\[
	h_{\mathrm{top}}(E,\mathcal{U},\{F_n\}_{n\in\mathbb N})=	\limsup_{n\to\infty}\frac{1}{|F_n|}\log N(\mathcal{U}^{F_n},E).
	\]
	The \emph{topological entropy of $E\subset X$ along the F\o lner sequence \(\{F_n\}_{n\in\mathbb N}\)} is defined as
	\[
	h_{\mathrm{top}}(E,\{F_n\}_{n\in\mathbb N})
	=
	\sup_{\mathcal{U}\in\mathfrak{C}_X^o}
	h_{\mathrm{top}}(E,\mathcal{U},\{F_n\}_{n\in\mathbb N}).
	\]
	
	We now turn to Bowen's dimensional entropy for subsets in amenable group actions, as introduced in \cite{Zheng-Chen2016}. Let  \(\emptyset \neq Z \subset X\), \(\mathcal{U} \in \mathfrak{C}_X^o\), and \(\lambda, \varepsilon > 0\). Put
	\[
	F_{\lambda, \varepsilon}(\mathcal{U}, Z, \{F_n\}_{n\in \mathbb{N}}) = \inf \sum_{i \in I} \exp(-\lambda |F_{n_i}|),
	\]
	where the infimum is taken over all countable families \(\{B_i\}_{i \in I}\) covering \(Z\) with each \(B_i \in \mathcal{U}^{F_{n_i}}\) for some \(|F_{n_i}| \ge \frac{1}{\varepsilon}\).
	Set $F_{\lambda,\varepsilon}(\mathcal{U},\emptyset,\{F_{n}\}_{n\in\mathbb{N}})=0$ by convention, and put
	\[
	F_{\lambda}(\mathcal{U},Z,\{F_{n}\}_{n\in\mathbb{N}})=\lim_{\varepsilon\to 0}F_{\lambda,\varepsilon}(\mathcal{U},Z,\{F_{n}\}_{n\in\mathbb{N}})=\sup_{\varepsilon>0}F_{\lambda,\varepsilon}(\mathcal{U},Z,\{F_{n}\}_{n\in\mathbb{N}}).
	\]
	In particular, $F_{\lambda}(\mathcal{U},Z,\{F_{n}\}_{n\in\mathbb{N}})\notin\{0,+\infty\}$ for at most one $\lambda>0$, so we may define
	\[
	\begin{aligned}
	h_{\mathrm{dim}}(Z,\mathcal{U},\{F_{n}\}_{n\in\mathbb{N}})&=\inf\{\lambda>0:F_{\lambda}(\mathcal{U},Z,\{F_{n}\}_{n\in\mathbb{N}})=0\}
		\\&=\sup\{\lambda>0:F_{\lambda}(\mathcal{U},Z,\{F_{n}\}_{n\in\mathbb{N}})=+\infty\},
	\end{aligned}
	\]
	where we set $\inf \emptyset=+\infty$ by convention. Thus for every $\lambda> h_{\mathrm{dim}}(Z,\mathcal{U},\{F_{n}\}_{n\in\mathbb{N}})$ one has $F_{\lambda}(\mathcal{U},Z,\{F_{n}\}_{n\in\mathbb{N}})=0$.
	Then the \emph{Bowen's dimensional entropy of \(Z\) along the F\o lner sequence $\{F_n\}_{n\in\mathbb{N}}$} is
	defined as
	\[
	h_{\mathrm{dim}}(Z, \{F_n\}_{n\in \mathbb{N}}) = \sup_{\mathcal{U} \in \mathfrak{C}_X^o} h_{\mathrm{dim}}(Z, \mathcal{U}, \{F_n\}_{n\in \mathbb{N}}).
	\]

The proof of our main theorems will require us to pass from \(\{F_n\}_{n\in\mathbb N}\)
	to a slightly enlarged F\o lner sequence. The next proposition shows that such a
	finite enlargement does not affect the topological or dimensional
	entropies of subsets with respect to a fixed open cover.
This invariance of entropy under finite F\o lner enlargement is a technical tool used throughout.
	
	\begin{prop}\label{prop16}
		Let $A\in\mathfrak{F}_G$ and $\emptyset\neq E\subset X$, $\mathcal{U}\in\mathfrak{C}_{X}^{o}$. Then
\begin{eqnarray*}
h_{\mathrm{top}}(E,\mathcal{U},\{F_{n}^{A}\}_{n\in \mathbb{N}}) & = & h_{\mathrm{top}}(E,\mathcal{U},\{F_{n}\}_{n\in \mathbb{N}}),\\ h_{\mathrm{dim}}(E,\mathcal{U},\{F_{n}^{A}\}_{n\in \mathbb{N}}) & = & h_{\mathrm{dim}}(E,\mathcal{U},\{F_{n}\}_{n\in \mathbb{N}}),
\end{eqnarray*}
and hence \(
h_{\mathrm{top}}(E,\{F_{n}^{A}\}_{n\in \mathbb{N}}) = h_{\mathrm{top}}(E,\{F_{n}\}_{n\in \mathbb{N}})\), \( h_{\mathrm{dim}}(E,\{F_{n}^{A}\}_{n\in \mathbb{N}}) = h_{\mathrm{dim}}(E,\{F_{n}\}_{n\in \mathbb{N}}).\)
	\end{prop}
	\begin{proof}
		Note that $\{F_{n}^{A}\}_{n\in \mathbb{N}}$ is still a F\o lner sequence by Lemma \ref{prop-folnerseq}, so the notations $h_{\mathrm{top}}(E,\mathcal{U},\{F_{n}^{A}\}_{n\in \mathbb{N}})$ and $h_{\mathrm{dim}}(E,\mathcal{U},\{F_{n}^{A}\}_{n\in \mathbb{N}})$ are well defined. Then the first identity of the conclusion comes directly from the following fact
		\[
		N(\mathcal{U}^{F_{n}},E)\leq N(\mathcal{U}^{F_{n}^{A}},E)\leq |\mathcal{U}|^{|A|}N(\mathcal{U}^{F_{n}},E).
		\]

Now let us prove firstly the direction $\le$ in the second identity. We pick arbitrarily
		$
		\lambda>h_{\mathrm{dim}}(E,\mathcal{U},\{F_n\}_{n\in\mathbb N}).
		$
		Then for every $\varepsilon>0$ one has
		\(
		F_{\lambda,\varepsilon}(\mathcal{U},E,\{F_n\}_{n\in\mathbb{N}})=0.
		\)
		Therefore, there exists a countable cover $\{Z_i\}_{i\in I}$ of $E$, with each \(Z_i \in \mathcal{U}^{F_{n_i}}\) for some \(|F_{n_i}| \ge \frac{1}{\varepsilon}\) such that
		$
		\sum_{i\in I}\mathrm{e}^{-\lambda|F_{n_i}|}\leq \varepsilon.
		$
Note that $Z_i\cap U\ (i\in I, U\in\mathcal{U}^A)$ forms a countable cover of $E$, and that by the above construction each $Z_i\cap U$ is contained in some element of $\mathcal{U}^{F_{n_i}^A}$. Thus
\begin{eqnarray*}
F_{\lambda,\varepsilon}(\mathcal{U},E,\{F_n^A\}_{n\in\mathbb N})
		& \leq & \sum_{i\in I}\sum_{U\in\mathcal{U}^A}\mathrm{e}^{-\lambda|F_{n_{i}}^{A}|} \\
& \leq & |\mathcal{U}|^{|A|}\sum_{i\in I}\mathrm{e}^{-\lambda|F_{n_i}^A|}
		\leq
		|\mathcal{U}|^{|A|}\sum_{i\in I}\mathrm{e}^{-\lambda|F_{n_i}|}
		\leq
		|\mathcal{U}|^{|A|}\cdot \varepsilon.
		\end{eqnarray*}
	By letting $\varepsilon\to 0$, one has
		$
		F_\lambda(\mathcal{U},E,\{F_n^A\}_{n\in\mathbb N})=0
		$
		and then
		$h_{\mathrm{dim}}(E,\mathcal{U},\{F_n^A\}_{n\in\mathbb N})\leq \lambda,$
		so $$
		h_{\mathrm{dim}}(E,\mathcal{U},\{F_n^A\}_{n\in\mathbb N})\leq
		h_{\mathrm{dim}}(E,\mathcal{U},\{F_n\}_{n\in\mathbb N}).
		$$

To finish the proof, it remains to prove the direction $\ge$ in the second identity. Fix any
		$
		\eta> h_{\mathrm{dim}}(E,\mathcal{U},\{F_n^A\}_{n\in\mathbb N}).
		$
		Then, for every $\varepsilon\in(0,|A|^{-1})$, one has
		\(
		F_{\eta,\varepsilon}(\mathcal{U},E,\{F_n^A\}_{n\in\mathbb{N}})=0,
		\)
		and so there exists a countable cover $\{W_j\}_{j\in J}$ of $E$ with each \(W_j \in \mathcal{U}^{F_{m_j}^A}\) for some \(|F_{m_j}^A| \ge \frac{1}{\varepsilon}\) such that
		$
		\sum_{j\in J}e^{-\eta|F_{m_j}^A|}\leq \varepsilon.
		$
		Note that in this case
		\[
		|F_{m_j}|\ge |F_{m_j}^A|-|A|\ge (\varepsilon^*)^{- 1}\quad \text{where}\quad \varepsilon^*\doteq \left(\frac{1-|A|\varepsilon}{\varepsilon}\right)^{- 1},
		\]
		moreover, each set $W_j$ is contained in some $C_j\in \mathcal{U}^{F_{m_j}}$ for every $j\in J$.
Thus, the countable family $\{C_j\}_{j\in J}$ is again a cover of $E$, and consequently
		$$
		F_{\eta,\varepsilon^*}(\mathcal{U},E,\{F_n\}_{n\in\mathbb N})
		\le
		\sum_{j\in J}e^{-\eta|F_{m_j}|}
		\le
		e^{|A|\eta}\sum_{j\in J}e^{-\eta|F_{m_j}^A|}
		\le
		e^{|A|\eta}\cdot \varepsilon.
		$$
		By taking $\varepsilon\to 0$ (and then $\varepsilon^*\to 0$) one obtains
		$
		F_\eta(\mathcal{U},E,\{F_n\}_{n\in\mathbb N})=0,
		$
and then $
		h_{\mathrm{dim}}(E,\mathcal{U},\{F_n\}_{n\in\mathbb N})\leq \eta
		$
and finally \(
		h_{\mathrm{dim}}(E,\mathcal{U},\{F_n\}_{n\in\mathbb N})\leq h_{\mathrm{dim}}(E,\mathcal{U},\{F_n^A\}_{n\in\mathbb N}).
		\)
	\end{proof}

	\section{Proof of Theorem \ref{thm2}}\label{conditional}
	
	In this section we prove Theorem \ref{thm2} relying crucially on the following key Proposition \ref{prop14} (that is, \cite[Proposition 4.1]{He-Zhang-Zhang2026-preprint}), which provides a
	uniform covering estimate needed to control the topological conditional entropy
	via the dimensional entropy of stable sets, and hence converts the local stable-set entropy information into
a global covering bound.

	\begin{prop}\label{prop14}
Let $\{F_{n}\}_{n\in\mathbb{N}}$ be a tempered F\o lner sequence satisfying the growth condition \eqref{1.1}, and $e_G\in F_n$ for every $n\in\mathbb N$. Let \(\mathcal{U} \in \mathfrak{C}_X^o\) and \(\eta > 0\). If $\lambda\in \mathbb{R}$ satisfies
		\[
		\sup_{x \in X} h_{\mathrm{dim}}(\Gamma_\eta(x), \mathcal{U}, \{F_n\}_{n\in \mathbb{N}}) < \lambda,
		\]
then, for each sufficiently small \(\delta > 0\), there exists \(P \in \mathbb{N}\) such that, for all \(n > P\),
		\[
		\sup_{x \in X} N(\mathcal{U}^{F_n}, B_\eta^{F_n}(x)) \le |\mathcal{U}|^{(\delta + \delta^{\frac{1}{4}})|F_n|}\cdot \mathrm{e}^{\lambda |F_n|} \cdot 2^{\delta |F_n|}.
		\]
	\end{prop}
	
Now we are ready to prove Theorem~\ref{thm2}.

	\begin{proof}[Proof of Theorem~\ref{thm2}]
		By \cite[Theorem~2.2 and Proposition~5.1]{Dou-Wang-Zhang2025}, one already has
		\begin{equation}\label{thm2ineq}
			h^{*}(X,G)\geq \lim_{\varepsilon\to 0}\sup_{x\in X}
			h_{\mathrm{dim}}(\Gamma_{\varepsilon}^{\{F_{n}\}_{n\in\mathbb{N}}}(x),\{F_{n}\}_{n\in\mathbb{N}})
			\geq \lim_{\varepsilon\to 0}\sup_{x\in X}
			h_{\mathrm{dim}}(\Gamma_{\varepsilon}(x),\{F_{n}\}_{n\in\mathbb{N}})
		\end{equation}
		for any F\o lner sequence. Therefore, it remains to prove
		\begin{equation}\label{conditional-estimate}
			h^{*}(X,G)\leq \lim_{\varepsilon\to 0}\sup_{x\in X}
			h_{\mathrm{dim}}(\Gamma_{\varepsilon}(x),\{F_{n}\}_{n\in\mathbb{N}})
		\end{equation}
		under the assumption that the F\o lner sequence $\{F_n\}_{n\in\mathbb{N}}$ is tempered and satisfies the growth condition (\ref{1.1}). We may assume that the right-hand side of (\ref{conditional-estimate}) is finite,
otherwise the conclusion is trivial.
		To prove (\ref{conditional-estimate}),
fix arbitrarily given
\[\lambda>\lim_{\varepsilon\to 0}\sup_{x\in X}h_{\mathrm{dim}}(\Gamma_{\varepsilon}(x),\{F_n\}_{n\in\mathbb N}),\]
it suffices to show \(h^{*}(X,G)\leq \lambda\).
Note that, by the definition of Bowen's dimensional entropy and Proposition \ref{prop16}, there exists \(\eta>0\) such that for every \(\mathcal U\in\mathfrak C_X^o\),
		\[
		\sup_{x\in X}h_{\mathrm{dim}}(\Gamma_{\eta}(x),\mathcal{U},\{\widetilde F_n\}_{n\in\mathbb N})\le \lambda,
		\]
		where $
		\widetilde F_n:=F_n\cup\{e_G\}$, $n\in\mathbb N.
		$ With the help of Lemma \ref{prop-folnerseq},  \(\{\widetilde F_n\}_{n\in\mathbb N}\) is still a tempered F\o lner
		sequence satisfying \eqref{1.1}. Now let $\mathcal U\in\mathfrak C_X^o$. By Proposition~\ref{prop14}, for every sufficiently small $\delta>0$, there exists $P\in\mathbb N$ such that for all $n>P$,
		\[
		\sup_{x\in X}N(\mathcal U^{\widetilde F_n},B_\eta^{\widetilde F_n}(x))
		\le
		|\mathcal U|^{(\delta+\delta^{\frac14})|\widetilde F_n|}\cdot
		\mathrm e^{\lambda|\widetilde F_n|}\cdot
		2^{\delta|\widetilde F_n|}.
		\]
		Choose $\mathcal V\in\mathfrak C_X^o$ with
		$
		\operatorname{diam}(\mathcal V):=\max\{\operatorname{diam}(V):V\in\mathcal V\}<\eta.
		$
		Then every element of \(\mathcal V^{\widetilde F_n}\) is contained in
	some Bowen ball	\(B_\eta^{\widetilde F_n}(x)\), and hence, for all $n>P$,
		\[
		N(\mathcal U^{\widetilde F_n}\mid \mathcal V^{\widetilde F_n})
		\le
		\sup_{x\in X}N(\mathcal U^{\widetilde F_n},B_\eta^{\widetilde F_n}(x))
		\le
		|\mathcal U|^{(\delta+\delta^{\frac14})|\widetilde F_n|}\cdot
		\mathrm e^{\lambda|\widetilde F_n|}\cdot
		2^{\delta|\widetilde F_n|}.
		\]
As the value $h(G,\mathcal U|\mathcal V)$ is independent of the F\o lner sequence,
by taking logarithms, dividing by \(|\widetilde F_n|\), taking the \(\limsup\) as \(n \to \infty\), and then letting \(\delta \to 0\), one has
		\[
		h(G,\mathcal U|\mathcal V)
		=
		\lim_{n\to\infty}
		\frac{1}{|\widetilde F_n|}
		\log N(\mathcal U^{\widetilde F_n}|\mathcal V^{\widetilde F_n})
		\le \lambda.
		\]
Since $\mathcal U\in\mathfrak C_X^o$ is arbitrary, it follows
	\[h^{*}(X,G)
		\le \sup_{\mathcal{U}\in \mathfrak{C}_{X}^{o}}h(G,\mathcal{U}|\mathcal{V}) \le
		\lambda,
		\]
and then one obtains easily the desired inequality \eqref{conditional-estimate}.
	\end{proof}

	\section{Proof of Theorem \ref{thm4}}\label{inequality for dim}
	
Having established the dimensional entropy characterization of topological conditional entropy via stable sets
in previous section, we now turn to the behavior of dimensional entropy under factor maps and prove Theorem~\ref{thm4}.

We may assume, without loss of generality, that $e_G\in F_n$ for all $n\in\mathbb{N}$. Otherwise, we set $F_n^*= F_n\cup \{e_G\}$ for each $n\in \mathbb{N}$, then $\{F_n^*\}_{n\in\mathbb N}$ is also a F\o lner sequence (Lemma~\ref{prop-folnerseq}). Instead what we shall prove is that
$$h_{\mathrm{dim}}(E,\{F_n^*\}_{n\in\mathbb N})
	\leq
	h_{\mathrm{dim}}(\pi(E),\{F_n^*\}_{n\in\mathbb N})
	+
	\sup_{y\in Y}h_{\mathrm{top}}(\pi^{-1}(y),\{F_n^*\}_{n\in\mathbb N}),
	$$
and then the required conclusion follows directly by applying Proposition \ref{prop16}.

\smallskip

The proof of Theorem~\ref{thm4}
proceeds by combining fibrewise covering
estimates (coming from the relative topological entropy assumption on the fibres)
with a dimensional entropy bound on the image set.
The new difficulty here is that we must simultaneously control the covering numbers on the fibres and dimensional entropy of the image set,
which requires a more delicate balance in the application of the Lindenstrauss covering lemma.
The key Proposition \ref{section-3} shows
how to combine these two pieces of information via Lemma \ref{coveringlem}.

\begin{prop} \label{section-3}
Let $\pi: (X, G)\to (Y, G)$ be a factor map between $G$-actions. Let \(\mathcal{U}
\in \mathfrak{C}_X^o\), \(\emptyset\neq E\subset X\) and real numbers $\beta, \tau$ satisfy \(\beta\ge 0, \tau > 0\). Assume that there is
\begin{enumerate}

\item a tempered F\o lner sequence $\{F_n^{(1)}\}_{n\in\mathbb N}$, which is a subsequence of \(\{F_n\}_{n\in \mathbb{N}}\), with constant $C> 1$, and

\item for each $K\in \N$ and \(\eta = \frac{\tau}{6}\), there exists a family $\{(V(y),R(y))\}_{y\in Y}$ of
open neighbourhoods $V(y)\subset Y$ and integers $R(y)\in\N$ satisfying that, for each $y\in Y$, $R (y)\ge K$ and $\pi^{-1}(V(y))$ can be covered by at most
$e^{(\beta+\eta)|F^{(1)}_{R(y)}|}$ elements of $\mathcal{U}^{F^{(1)}_{R(y)}}$.
\end{enumerate}
Let $\lambda\in \mathbb{R}$ satisfy \(h_{\mathrm{dim}}(\pi (E), \{F_n\}_{n\in \mathbb{N}}) < \lambda\). Then
\[F_{\lambda+\beta+\tau}(\mathcal{U},E,\{F_{n}\}_{n\in\mathbb{N}})= 0.\]
\end{prop}

	\begin{proof}
Let $\delta>0$ be small enough such that
		\begin{equation}\label{5.2}
			10\delta^{\frac{1}{4}}<\frac{1}{2},\quad 10\delta^{\frac{1}{4}}\beta<\frac{1}{6}\tau,\quad\text{and}\quad |\mathcal{U}|^{(\delta+\delta^{\frac{1}{4}})}<\mathrm{e}^{\frac{1}{3}\tau},
		\end{equation}
		and let $M$ be the integer in Lemma \ref{coveringlem} according to $\delta,\ C$ and $D=\{e_{G}\}$.

\smallskip
	
To apply the Lindenstrauss covering lemma (Lemma~\ref{coveringlem}), let us define the sets $H_{i,j}$.

		For $i=1$, note that by the assumption we have obtained a family $\{(V(y),R(y))\}_{y\in Y}$ as above. By the compactness of $Y$, there exists $\{y_{1,1},\cdots,y_{1,P_1}\}\subset Y$ such that
		$
		Y=\bigcup_{q=1}^{P_1}V(y_{1,q}).
		$
		For simplicity, we write
		$
		V_{1,q}:=V(y_{1,q})$ for $q=1,\cdots,P_1
		$
		and set
		$
		\mathcal{V}_1=\{V_{1,1},\cdots,V_{1,P_1}\}.
		$
		Collecting the corresponding finitely many integers and finitely many open subsets associated with the points $y_{1,1},\cdots,y_{1,P_1}$, we obtain the following:
		\begin{itemize}
			\item There exist finitely many integers $R_{1,1}<\cdots<R_{1,N_{1}}$.
			
			\item For each $q\in \{1,\cdots,P_{1}\}$, there exists $j=j(1,q)\in\{1,\cdots,N_{1}\}$ such that $\pi^{-1}(V_{1,q})$
is covered by at most $\mathrm{e}^{(\beta+\eta)|F_{R_{1,j}}^{(1)}|}$ elements in $\mathcal{U}^{F_{R_{1,j}}^{(1)}}$.
		\end{itemize}
		
	For $i=2,\cdots,M$, similarly, there exists a finite subset $\{y_{i,1},\cdots,y_{i,P_i}\}\subset Y$, and a finite open cover  $\mathcal{V}_{i}=\{V_{i,1},\cdots,V_{i,P_i}\}$ of $Y$ with $y_{i,q}\in V_{i,q}$ for each $q=1,\cdots,P_i$, furthermore:
		\begin{itemize}
			\item There exist finitely many integers $R_{i,1}<\cdots<R_{i,N_{i}}$.
			
			\item For each $q\in \{1,\cdots,P_{i}\}$, there exists $j=j(i,q)\in\{1,\cdots,N_{i}\}$ such that $\pi^{-1}(V_{i,q})$ is covered by at most $\mathrm{e}^{(\beta+\eta)|F_{R_{i,j}}^{(1)}|}$ elements in $\mathcal{U}^{F_{R_{i,j}}^{(1)}}$.
		\end{itemize}
		Note that, $R(y)\in\mathbb N$ in the second item of the assumption can be chosen sufficiently large for each $y\in Y$ and that $\{F_n^{(1)}\}_{n\in\mathbb N}$ is a F\o lner sequence, we may assume additionally that
		\begin{equation}\label{construction 2}
			R_{i-1,N_{i-1}}<R_{i,1}
			\quad\text{and}\quad\left|\bigcup_{i'<i}\bigcup_{j=1}^{N_{i'}}\{e_{G}\}{F_{R_{i',j}}^{(1)}}^{-1}F_{R_{i,k}}^{(1)}\right|\leq (1+\delta)|F_{R_{i,k}}^{(1)}|
		\end{equation}
		for each $i=2,\cdots,M$ and every $k= 1,\cdots,N_{i}$.
		
		The required $H_{i,j}$ is defined as $F_{R_{i,j}}^{(1)}$ for each $i=1,\cdots,M$ and every $j= 1,\cdots,N_{i}$.

\smallskip		
		
		We now turn back to the original F\o lner sequence $\{F_n\}_{n\in\mathbb N}$. By the definition of F\o lner sequences, there exists $N\in \mathbb{N}$ (depending on the previously constructed $H_{i,j}$ and $\delta$) such that for all $n\geq N$, one has $|F_{n}^{*}|\geq (1-\delta)|F_{n}|$, where
		\[
		F_{n}^{*}=\{f\in F_{n}:H_{i,j}f\subset F_{n}\,\,\text{for each}\,\, i=1,\cdots,M\,\,\text{and every}\,\, j= 1,\cdots,N_{i}\}.
		\]
		
		Let
		$
		\mathcal{V}=\bigvee_{i=1}^{M}\mathcal{V}_{i}.$ Recall that we have assumed that $h_{\mathrm{dim}}(\pi(E),\{F_{n}\}_{n\in\mathbb{N}})<\lambda$, and hence
		$
		h_{\mathrm{dim}}(\pi(E),\mathcal{V},\{F_{n}\}_{n\in\mathbb{N}})<\lambda.
		$
		In particular, for every $\varepsilon\in(0,{(|F_{1}|+\cdots+|F_{N}|)}^{-1})$, we have \(F_{\lambda,\varepsilon}(\mathcal{V},\pi(E),\{F_{n}\}_{n\in\mathbb{N}})=0.\)
		Thus there exists a countable cover $\{E_{k}\}_{k\in I}$ of $\pi(E)$ such that
		\(\sum_{k\in I}\mathrm{e}^{-\lambda|F_{m_{k}}|}\leq \varepsilon\) and each $E_k$ is an element of $\mathcal{V}^{F_{m_k}}$ for some $|F_{m_k}|\geq \frac{1}{\varepsilon}$.

\smallskip

Fix $E_{k}$ for some $k\in I$. Now let us estimate $N(\mathcal{U}^{F_{m_{k}}},\pi^{-1}E_k)$.

		For each $i\in\{1,\cdots,M\}$ and every $ g \in F_{m_{k}}^{*}$, we have
		\begin{equation}\label{5.3}
			gE_{k}\in g\mathcal{V}^{F_{m_{k}}}\succeq g\mathcal{V}^{F_{R_{1,1}}^{(1)}g}=\mathcal{V}^{F_{R_{1,1}}^{(1)}}\succeq \mathcal{V}\succeq \mathcal{V}_{i},
		\end{equation}
		where the first ``$\succeq$'' holds due to $F_{R_{1,1}}^{(1)}g=H_{1,1}g\subset F_{m_{k}}$, and the second ``$\succ$'' comes from $e_{G}\in F_{n}^{(1)}\  \text{for all }n\in \mathbb{N}$ (we have made such assumption at the beginning of the section). We deduce from (\ref{5.3}) that,
		$g\pi^{-1}(E_{k}) $ is contained in some element of $\pi^{-1}\mathcal{V}_{i}$,
		and hence
		$$
		g\pi^{-1}(E_{k})\text{ can be covered by at most } \mathrm{e}^{(\beta+\eta)|F_{R_{i,i(g)}}^{(1)}|} \text{ elements in } \mathcal{U}^{F_{R_{i,i(g)}}^{(1)}}
		$$
		for some $i(g)\in \{1,\cdots,N_{i}\}$ by the construction of $\mathcal{V}_{i}$.
		Now for each $j=1,\cdots,N_{i}$ we set
		\[
		A_{i,j}(E_{k})=\left\{g\in F_{m_{k}}^{*}: N \left( \mathcal{U}^{F_{R_{i,j}}^{(1)}}, g\pi^{-1}(E_{k})\right)\le \mathrm{e}^{(\beta+\eta)|F_{R_{i,j}}^{(1)}|}\right\}.
		\]
		Then by the above discussion,
		\[
		F_{m_{k}}^{*}=\bigcup_{j=1}^{N_{i}}A_{i,j}(E_{k}) \ \ \text{and}\ \  H_{i,j}A_{i,j}(E_{k})\subset F_{m_{k}}.
		\]
		Consider the collection $\{H_{i,j}a:a\in A_{i,j}(E_{k});i=1,\cdots,M;\, j=1,\cdots,N_{i}\}$. Since $|F_{m_{k}}|\geq \frac{1}{\varepsilon}$ and hence $m_{k}\geq N$, the assumptions in Lemma \ref{coveringlem} are satisfied with $\alpha=1-\delta$. Then the collection admits a $10\delta^\frac{1}{4}$-disjoint subfamily $\mathcal{F}(E_k)$ with $|\cup \mathcal{F}(E_k)|\geq (1-\delta-\delta^{\frac{1}{4}})|F_{m_{k}}|$.

		Note that, for each $F\in \mathcal{F}(E_k)$, it can be written in the form
		\(
		F=H_{i,j}a \,(=F_{R_{i,j}}^{(1)}a)
		\)
		for some $i=i_F\in \{1,\cdots,M\}$, $j=j_F\in \{1,\cdots,N_i\}$, and $a=a_{F}\in A_{i,j}(E_k)$. Thus
\[
N \left(\mathcal{U}^{F_{R_{i_F,j_F}}^{(1)}}, a_{F}\pi^{-1}(E_{k})\right)\le \mathrm{e}^{(\beta+\eta)|F_{R_{i_F,j_F}}^{(1)}|}\quad \text{equivalently}
\quad N (\mathcal{U}^{F}, \pi^{-1}(E_{k}))\le \mathrm{e}^{(\beta+\eta)|F|}.
\]

It follows directly that
		\[
		\begin{aligned}
N (\mathcal{U}^{\cup\mathcal{F} (E_k)}, \pi^{-1}(E_k)) & \le \prod_{F\in \mathcal{F}(E_k)} N (\mathcal{U}^{F}, \pi^{-1}(E_{k})) \le \exp\left((\beta+\eta)\sum_{F\in \mathcal{F}(E_k)}\left|F \right|\right)
			\\&\leq \exp\left((\beta+\eta)\frac{|\cup\mathcal{F} (E_k)|}{1-10\delta^{\frac{1}{4}}}\right)\ (\text{as }\mathcal{F} (E_k)\text{ is }10\delta^{\frac{1}{4}}\text{-disjoint})
			\\&\leq \exp\left((\beta+\eta)\frac{|F_{m_{k}}|}{1-10\delta^{\frac{1}{4}}}\right),
		\end{aligned}
		\]
and then
\begin{equation}\label{5.4}
			\begin{aligned}
				N(\mathcal{U}^{F_{m_{k}}},\pi^{-1}(E_{k}))
				&\leq|\mathcal{U}|^{|F_{m_{k}}\setminus \cup\mathcal{F} (E_k)|}\cdot N(\mathcal{U}^{\cup\mathcal{F} (E_k)},\pi^{-1}(E_{k}))
				\\&\leq \Big( |\mathcal{U}|^{(\delta+\delta^{\frac{1}{4}})}\cdot \exp \Big((\beta+\eta)\frac{1}{1-10\delta^{\frac{1}{4}}} \Big)   \Big)^{|F_{m_{k}}|}.
			\end{aligned}
		\end{equation}

Recall that the countable family $\{E_{k}\}_{k\in I}$ covers $\pi(E)$ satisfying
		\[
		\sum_{k\in I}\mathrm{e}^{-\lambda|F_{m_{k}}|}\leq \varepsilon\quad \text{and}\quad E_{k}\in  \mathcal{V}^{F_{m_{k}}}\ \text{and}\ |F_{m_{k}}|\geq \frac{1}{\varepsilon}\ \text{for each}\ k\in I.
		\]
		Then, $\{\pi^{- 1} (E_{k})\}_{k\in I}$ is a countable family covering $E$, and so by noting $\eta = \frac{\tau}{6}$ one has
	\[
		\begin{aligned}
			F_{\lambda+\beta+\tau,\varepsilon}(\mathcal{U},E,\{F_{n}\}_{n\in\mathbb{N}})&\leq \sum_{k\in I}N(\mathcal{U}^{F_{m_{k}}},\pi^{-1}(E_{k}))\cdot \mathrm{e}^{-(\lambda+\beta+\tau)|F_{m_{k}}|}
			\\&\leq \sum_{k\in I}\Big( |\mathcal{U}|^{(\delta+\delta^{\frac{1}{4}})}\cdot \exp \Big((\beta+\eta)\frac{1}{1-10\delta^{\frac{1}{4}}} \Big)\cdot\mathrm{e}^{-(\lambda+\beta+\tau)}\Big)^{|F_{m_{k}}|}
			\\&\leq \sum_{k\in I}\mathrm{e}^{-\lambda|F_{m_{k}}|}\leq\varepsilon,
		\end{aligned}
		\]
		where the second and third inequalities follows from (\ref{5.4}) and (\ref{5.2}), respectively. Thus it follows from the definition that
		\(
		F_{\lambda+\beta+\tau}(\mathcal{U},E,\{F_{n}\}_{n\in\mathbb{N}})=0.
		\)
		This ends the proof.
	\end{proof}

	We can now proceed to the proof of Theorem \ref{thm4}.
	
	\begin{proof}[Proof of Theorem \ref{thm4}]
It makes no difference to assume that $$h_{\mathrm{dim}}(\pi(E),\{F_{n}\}_{n\in \mathbb{N}})<\infty \ \ \text{and}\ \ \sup_{y\in Y}h_{\mathrm{top}}(\pi^{-1}(y),\{F_{n}\}_{n\in \mathbb{N}})\ (\text{denoted by}\ \beta)<\infty$$ otherwise there is nothing to prove. Fix arbitrarily given $\mathcal{U}\in\mathfrak{C}_{X}^{o}$ and $\lambda, \tau\in \mathbb{R}$ satisfying $\lambda>h_{\mathrm{dim}}(\pi(E),\{F_{n}\}_{n\in \mathbb{N}})$ and $\tau>0$.
		It suffices to show that
		\begin{equation}\label{key goal of thm}
			F_{\lambda+\beta+\tau}(\mathcal{U},E,\{F_{n}\}_{n\in\mathbb{N}})=\lim_{\varepsilon\to 0}F_{\lambda+\beta+\tau,\varepsilon}(\mathcal{U},E,\{F_{n}\}_{n\in\mathbb{N}})=0.
		\end{equation}

		Recall that no temperedness assumption is imposed on $\{F_n\}_{n\in\mathbb N}$ in Theorem~\ref{thm4}. Nevertheless, by taking a subsequence we may take a tempered F\o lner subsequence $\{F_n^{(1)}\}_{n\in\mathbb N}$.

Fix arbitrarily $K\in \N$ and set \(\eta = \frac{\tau}{6}\). Note that for each $y \in Y$ we have
		\[
		\limsup_{n\to \infty}\frac{1}{|F_{n}^{(1)}|}\log N(\mathcal{U}^{F_{n}^{(1)}},\pi^{-1}(y))\leq \limsup_{n\to \infty}\frac{1}{|F_{n}|}\log N(\mathcal{U}^{F_{n}},\pi^{-1}(y))\leq \beta.
		\]
Then there exists $R(y)\in\mathbb N$ with $R (y)\ge K$, such that
		\[
			N(\mathcal{U}^{F_{R(y)}^{(1)}},\pi^{-1}(y))\leq \mathrm{e}^{(\beta+\eta)|F_{R(y)}^{(1)}|},
		\]
and so $\pi^{-1}(y)$ is covered by an open subset $U(y)$ with
\[N(\mathcal{U}^{F_{R(y)}^{(1)}}, U(y))\leq \mathrm{e}^{(\beta+\eta)|F_{R(y)}^{(1)}|}.\]
 Thus there exists an open neighborhood of $y$, say $V(y)$, with $\pi^{-1}(V(y))\subset U(y)$, and then $\pi^{-1}(V(y))$ can be covered by at most
$e^{(\beta+\eta)|F^{(1)}_{R(y)}|}$ elements of $\mathcal{U}^{F^{(1)}_{R(y)}}$.
	That is, the assumption of Proposition \ref{section-3} is satisfied, and then the required identity \eqref{key goal of thm} comes directly.
	\end{proof}
	
	\section{Proof of Theorem \ref{thm3}}\label{relative}
	
	In this section, we prove Theorem \ref{thm3}, which characterizes relative topological entropy via Bowen's dimensional entropy along the fibres.
The proof of Theorem \ref{thm3} follows a similar strategy to that of Theorem \ref{thm4} but with a key difference: instead of combining a dimensional entropy bound on the image with fibre topological entropy bounds, we now need to work directly with the dimensional entropy of the fibres.

\smallskip

The proof of Theorem \ref{thm3} relies again on a key covering proposition, Proposition \ref{prop15}, whose statement and proof are inspired by \cite[Proposition 2.6]{Dou-Wang-Zhang2025} and Propositions~\ref{prop14} and \ref{section-3}.
We provide here a proof of Proposition \ref{prop15} for completeness.

	\begin{prop}\label{prop15}
	Let $\{F_{n}\}_{n\in\mathbb{N}}$ be a tempered F\o lner sequence
		with constant $C$, satisfying the growth condition \eqref{1.1} and $e_G\in F_n$ for every $n\in\mathbb N$. Let
		\(\mathcal{U} \in \mathfrak{C}_X^o\). If $\lambda\in \mathbb{R}$ satisfies
		\[
		\sup_{y\in Y} h_{\mathrm{dim}}(\pi^{-1}(y),\mathcal U,\{F_n\}_{n\in\mathbb N})<\lambda,
		\]
		then, for every sufficiently small \(\delta>0\), there exists \(P\in\mathbb N\) such that, for all \(n>P\),
		\[
		\sup_{y\in Y} N(\mathcal U^{F_n},\pi^{-1}(y))
		\le
		|\mathcal U|^{(\delta+\delta^{\frac{1}{4}})|F_n|}\cdot
		\mathrm e^{\lambda|F_n|}\cdot
		2^{\delta|F_n|}.
		\]
	\end{prop}
	\begin{proof}
		By choosing $\delta\in (0,20^{-4})$ small enough, we pick $\lambda^{*}\in \mathbb{R}$ such that
		\begin{equation} \label{estimate-fibredim}
			\lambda_\delta:=(1-10\delta^{1/4})\lambda
			>
			\lambda^*
			>
			\sup_{y\in Y}
			h_{\mathrm{dim}}(\pi^{-1}(y),\mathcal U,\{F_n\}_{n\in\mathbb N}).
		\end{equation}

		We take $M\in \mathbb{N}$ large enough (depending on $\delta$, $C$ and $D=\{e_{G}\}$) by Lemma \ref{coveringlem}, we also take $N\in \mathbb{N}$ by Lemma \ref{combilem} according to $\delta$.
		Since \(G\) is infinite and \(\{F_n\}_{n\in\mathbb N}\) is a F\o lner sequence,
		we have \(|F_n|\to\infty\). Hence, there exists $N'\in\mathbb{N}$ such that for every \(n\ge N'\),
		\begin{equation}\label{expo}
			|F_n|\leq \mathrm{e}^{(\lambda_{\delta}-\lambda^{*})|F_n|}.
		\end{equation}
		We shall denote from now on $\widetilde{N}=\max\{N,N'\}$ for simplicity.

		By the construction (\ref{estimate-fibredim}), for $\varepsilon_1=(|F_{1}|+\cdots+|F_{\widetilde N}|)^{-1}>0$ and each $y\in Y$, one has
		\begin{equation}\label{4.2}
			0=F_{\lambda^{*}}(\mathcal{U},\pi^{-1}(y),\{F_{n}\}_{n\in\mathbb{N}})\geq F_{\lambda^{*},\varepsilon_1}(\mathcal{U},\pi^{-1}(y),\{F_{n}\}_{n\in\mathbb{N}}),
		\end{equation}
		and then, by the compactness of $\pi^{-1}(y)$, there exist
		\begin{itemize}
			\item finitely many integers
			$
			(\widetilde N\leq)\ R_{y,1}<\cdots<R_{y,N_{y}},
			$ and
			\item finitely many open subsets
			$
			\{B_{y,1,r}\}_{r=1}^{k(y,1)}\subset \mathcal{U}^{F_{R_{y,1}}},\cdots,\{B_{y,N_{y},r}\}_{r=1}^{k(y,N_{y})}\subset \mathcal{U}^{F_{R_{y,N_{y}}}}
			$,
		\end{itemize}
		such that
		\[
		\pi^{-1}(y)\subset U(y):=\bigcup_{j=1}^{N_{y}}\bigcup_{r=1}^{k_(y,j)}B_{y,j,r} \ \ \text{and}\ \ \sum_{j=1}^{N_{y}}k(y,j)\mathrm{e}^{-\lambda^{*}|F_{R_{y,j}}|}<1.
		\]
Thus there exists an open set $V (y)$ containing $y$ with $\pi^{-1}(V(y))\subset U(y)$.
		After merging possible repetitions among the sets \(F_{R_{y,j}}\), we assume
$F_{R_{y,1}},\cdots,F_{R_{y,N_y}}$
		are pairwise distinct.

\smallskip

		To apply Lindenstrauss covering lemma (Lemma \ref{coveringlem}), let us construct $H_{i,j}$ as follows.

		For $i=1$, note that we have obtained an open cover $\{V(y):y\in Y\}$ as above. By the compactness of $Y$, there exist $\{y_{1,1},\cdots,y_{1,P_{1}}\}$ such that $Y\subset \bigcup_{q=1}^{P_{1}}V(y_{1,q})$. We shall rename $V(y_{1,q})$ to $V_{1,q}$ for each $q=1,\cdots,P_1$ for simplicity. By rewriting those finitely many integers and finitely many open subsets for all points $y_{1,1},\cdots,y_{1,P_1}$ one has that:
		
		\begin{itemize}
			\item There exist finitely many integers $(\widetilde N\leq)\ R_{1,1}<\cdots<R_{1,N_{1}}$ with $F_{R_{1,1}},\cdots,F_{R_{1,N_1}}$ pairwise distinct.
			
			\item For each $q\in\{1,\cdots,P_{1}\}$, there are finitely many open subsets of $X$
			$$
			\{B_{1,q,r}^{(1)}\}_{r=1}^{k_{1}(1,q)}\subset \mathcal{U}^{F_{R_{1,1}}},\cdots,\{B_{N_{1},q,r}^{(1)}\}_{r=1}^{k_{1}(N_{1},q)}\subset \mathcal{U}^{F_{R_{1,N_{1}}}},
			$$
			some of $k_{1}(j,q)$ for $j=1,\cdots,N_1$ may be zero, such that
			$$
			\pi^{-1}(y_{1,q})\subset\pi^{-1}(V_{1,q})\subset U_{1,q}:=\bigcup_{j=1}^{N_{1}}\bigcup_{r=1}^{k_{1}(j,q)}B_{j,q,r}^{(1)}\ \ \text{and}\ \  \,\sum_{j=1}^{N_{1}}k_{1}(j,q)\mathrm{e}^{-\lambda^{*}|F_{R_{1,j}}|}<1.
			$$
		\end{itemize}

		For $i=2,\cdots,M$, as was done in the previous paragraphs (the only difference being that we shall consider $F_{\lambda^{*},\varepsilon_i}$ for $\varepsilon_i$ instead of $\varepsilon_1$, where $\varepsilon_i$ will be specified later), there exists a finite subset $\{y_{i,1},\cdots,y_{i,P_i}\}\subset Y$, and a finite open cover $\{ V_{i,1},\cdots,V_{i,P_{i}}\}$ of $Y$ with $y_{i,q}\in V_{i,q}$ for each $q=1,\cdots,P_i$, furthermore:
		\begin{itemize}
			\item There exist finitely many integers
			$
			R_{i,1}<\cdots<R_{i,N_{i}},
			$ with each $|F_{R_{i,j}}|\geq \frac{1}{\varepsilon_i}$, and the sets $F_{R_{i,1}},\cdots,F_{R_{i,N_i}}$ are pairwise distinct.
			
			\item For each $q\in\{1,\cdots ,P_{i}\}$, there exist finitely many open subsets of $X$
			$$
			\{B_{1,q,r}^{(i)}\}_{r=1}^{k_{i}(1,q)}\subset \mathcal{U}^{F_{R_{i,1}}},\cdots,\{B_{N_{i},q,r}^{(i)}\}_{r=1}^{k_{i}(N_{i},q)}\subset \mathcal{U}^{F_{R_{i,N_{i}}}},
			$$
			some of $k_{i}(j,q)$ for $j=1,\cdots,N_i$ may be zero, such that
			$$
			\pi^{-1}(y_{i,q})\subset \pi^{-1}(V_{i,q})\subset U_{i,q}:=\bigcup_{j=1}^{N_{i}}\bigcup_{r=1}^{k_{i}(j,q)}B_{j,q,r}^{(i)} \ \ \text{and}\ \ \,\sum_{j=1}^{N_{i}}k_{i}(j,q)\mathrm{e}^{-\lambda^{*}|F_{R_{i,j}}|}<1.
			$$
		\end{itemize}
		Note that $G$ is assumed to be a countably infinite discrete amenable group and $\{F_n\}_{n\in\mathbb{N}}$ is a F\o lner sequence. By choosing $\varepsilon_i>0$ small enough, we may assume additionally that
		\begin{equation}\label{construction}
			\max_{p= 1}^{N_{i-1}} |F_{R_{i-1,p}}| < \min_{q= 1}^{N_i} |F_{R_{i,q}}|\ \ \text{and}\ \ \left|\bigcup_{i'<i}\bigcup_{j=1}^{N_{i'}}\{e_{G}\}F_{R_{i',j}}^{-1}F_{R_{i,k}}\right|\leq (1+\delta)|F_{R_{i,k}}|
		\end{equation}
		for each $i=2,\cdots,M$ and every $k= 1,\cdots,N_{i}$.

		The required $H_{i,j}$ is defined as $F_{R_{i,j}}$ for each  $i=1,\cdots,M$ and $j= 1,\cdots,N_{i}$.
		
	\smallskip
		
		Since $\{F_{n}\}_{n\in\mathbb{N}}$ is a F\o lner sequence, there exists $Q_{1}\in \mathbb{N}$ (depending on the previously constructed $H_{i,j}$ and $\delta$) such that for all $n\geq Q_{1}$, we have $|F_{n}^{*}|\geq (1-\delta)|F_{n}|$, where
		\[
		F_{n}^{*}=\{f\in F_{n}:H_{i,j}f\subset F_{n}\,\,\text{for each}\,\, i=1,\cdots,M\,\,\text{and every}\,\, j= 1,\cdots,N_{i}\}.
		\]
		
		Fix any sufficiently large $n>Q_{1}$ and any $y\in Y$; we now estimate $N(\mathcal{U}^{F_{n}},\pi^{-1}(y))$.

		For each $i\in\{1,\cdots,M\}$ and every $g \in F_{n}^{*}$, notice that $\{V_{i,j}\}_{j=1}^{P_{i}}$ is an open cover of $Y$, then $gy\in V_{i,i(g)}$ for some $i(g)\in \{1,\cdots,P_{i}\}$, and thus
		\[
		\pi^{-1}(gy)\subset \pi^{-1}(V_{i,i(g)})\subset U_{i,i(g)}= \bigcup_{j=1}^{N_{i}}\bigcup_{r=1}^{k_{i}(j,i(g))}B_{j,i(g),r}^{(i)}.
		\]

		Pick arbitrarily $z\in \pi^{-1}(y)$ (so
		$gz\in g\pi^{-1}(y)=\pi^{-1}(gy)$).
		For each $j=1,\cdots,N_i$ we set
		\[
		A_{i,j}(z)=\left\{g\in F_{n}^{*}:gz\in \bigcup_{r=1}^{k_{i}(j,i(g))}B_{j,i(g),r}^{(i)}\right\}.
		\]
		It should be noted that  $i(g)$ depends only on $i,g$ and $y$, and is independent of $z$. Then
		\[
		F_{n}^{*}=\bigcup_{j=1}^{N_{i}}A_{i,j}(z) \,\,\text{ and }\,\, H_{i,j}A_{i,j}(z)\subset F_{n}.
		\]
		Consider the collection $\{H_{i,j}a:a\in A_{i,j}(z);i=1,\cdots,M;\, j=1,\cdots,N_{i}\}$. It admits a $10\delta^\frac{1}{4}$-disjoint sub-collection $\mathcal{F}_{z}$ satisfying $|\cup \mathcal{F}_{z}|\geq (1-\delta-\delta^{\frac{1}{4}})|F_{n}|$ by Lemma \ref{coveringlem}.

		Let $\mathfrak{F}=\{\mathcal{F}_{z}:z\in \pi^{-1}(y)\}$. Note that each family $\mathcal{F}_{z}$ is a $10\delta^{\frac{1}{4}}$-disjoint, and hence $\frac{1}{2}$-disjoint because $\delta\in (0,20^{-4})$, sub-collection of $\{F_{R_{i,j}}a:F_{R_{i,j}}a\subset F_{n};\,i=1,\cdots,M;\, j= 1,\cdots,N_{i}\}$ (recall $H_{i,j}=F_{R_{i,j}}$), and that $N\in\mathbb{N}$ is constructed by Lemma \ref{combilem} according to $\delta$. Using Lemma \ref{combilem} one obtains that, there exists $Q_{2}\in \mathbb{N}$ (depending on previously constructed $R_{i,j}$) such that $|\mathfrak{F}|\leq 2^{\delta|F_{n}|}$ for $n\geq Q_{2}$.

		Fix arbitrarily given $\mathcal{F}\in \mathfrak{F}$ and define $W_{\mathcal{F}}=\{z\in \pi^{-1}(y):\mathcal{F}_{z}=\mathcal{F}\}$. We shall prove
		\begin{equation}\label{estimate-number of disjoint}
			N(\mathcal{U}^{F_n},W_{\mathcal{F}})\leq \mathrm{e}^{\lambda|F_n|}\cdot|\mathcal{U}|^{(\delta+\delta^{\frac{1}{4}})|F_n|}.
		\end{equation}
Once this is proved, combined with the bound on \(|\mathfrak{F}|\), one has that
		\[
		N(\mathcal{U}^{F_n}, \pi^{-1}(y)) \le \sum_{\mathcal{F} \in \mathfrak{F}} N(\mathcal{U}^{F_n}, W_\mathcal{F}) \le |\mathcal{U}|^{(\delta+\delta^{\frac{1}{4}})|F_n|}\cdot e^{\lambda |F_n|}\cdot 2^{\delta |F_n|}
		\]
		whenever $n>P=\max\{Q_{1},Q_{2}\}$, completing the proof of the conclusion.

\smallskip

Now let us prove \eqref{estimate-number of disjoint}. Take arbitrarily \(F\in \mathcal{F}\). Then \(F\) is exactly
		\(F_{R_{i,j}}a\) for some \(i\in \{1,\cdots,M\}\),
		\(j\in\{1,\cdots,N_i\}\), and \(a\in F_n^*\). In fact, such \(i\)
		exists uniquely by \eqref{construction}; denote it by \(i_F\).
		Set
		$
		\mathcal{E}_F
		:=
		\{
		a\in F_n^*:
		\text{ there exists }j\in\{1,\cdots,N_{i_F}\}
		\text{ such that }F=F_{R_{i_F,j}}a
		\}.
		$
		Recall that for each \(i\), the sets
		\(F_{R_{i,1}},\cdots,F_{R_{i,N_i}}\) are pairwise distinct. Thus for every
		\(a\in \mathcal{E}_F\), there exists a unique index \(j_F(a)\in\{1,\cdots,N_{i_F}\}\)
		such that
		$
		F=F_{R_{i_F,j_F(a)}}a.
		$
		We now claim
		\begin{equation}\label{estimate WF}
			W_{\mathcal F}
			\subset
			\bigcup_{a\in \mathcal{E}_F}
			\bigcup_{r=1}^{k_{i_F}(j_F(a),i_F(a))}
			a^{-1}B_{j_F(a),i_F(a),r}^{(i_F)} .
		\end{equation}
		Indeed, take arbitrarily \(z\in W_{\mathcal F}\). Since \(\mathcal F_z=\mathcal F\)
		and \(F\in\mathcal F\), by the construction of \(\mathcal F_z\) there
		exist \(j\in\{1,\cdots,N_{i_F}\}\) and \(a\in A_{i_F,j}(z)\) such that
		$
		F=F_{R_{i_F,j}}a,
		$
		and hence \(a\in \mathcal{E}_F\) and \(j=j_F(a)\). Then (\ref{estimate WF}) follows from the definition of
		\(A_{i_F,j_F(a)}(z)\), as
		\[
		az\in
		\bigcup_{r=1}^{k_{i_F}(j_F(a),i_F(a))}
		B_{j_F(a),i_F(a),r}^{(i_F)},\quad\text{equivalently,}\quad z\in
		\bigcup_{r=1}^{k_{i_F}(j_F(a),i_F(a))}
		a^{-1}B_{j_F(a),i_F(a),r}^{(i_F)}.
		\]
		 Note that the sets appearing in the right-hand side of (\ref{estimate WF})
		are elements of \(\mathcal U^F\) since \(F=F_{R_{i_F,j_F(a)}}a\).
		We also note that \(\mathcal E_F\) has cardinality at most \(|F|\). In fact,
		since \(e_G\in F_{R_{i_F,j_F(a)}}\), every \(a\in\mathcal E_F\) satisfies
		$
		a\in F_{R_{i_F,j_F(a)}}a=F.
		$
		Thus \(\mathcal E_F\subset F\), and hence
		$
		|\mathcal E_F|\le |F|.
		$
		Therefore
		\begin{equation}\label{dk}
			\begin{aligned}
				N(\mathcal U^F,W_{\mathcal F})
				&\le
				\sum_{a\in \mathcal{E}_F}
				k_{i_F}(j_F(a),i_F(a))\ \ \ (\mbox{by}\ (\ref{estimate WF}))
				\\&= \left(\sum_{a\in \mathcal{E}_F}k_{i_F}(j_F(a),i_F(a))\mathrm{e}^{-\lambda^{*}|F_{R_{i_F,j_F(a)}}|}\right)\cdot \mathrm{e}^{\lambda^{*}|F|} \\&<|\mathcal{E}_{F}|\cdot  \mathrm{e}^{\lambda^{*}|F|}\ \ \ (\mbox{as}\,\, k_{i_F}(j_F(a),i_F(a))\mathrm{e}^{-\lambda^{*}|F_{R_{i_F,j_F(a)}}|}<1)
				\\&\leq |F|\cdot \mathrm{e}^{(\lambda^{*}-\lambda_{\delta})|F|}\cdot \mathrm{e}^{\lambda_{\delta}|F|}\leq \mathrm{e}^{\lambda_{\delta}|F|}\ \ \ (\text{by}\ (\ref{expo})\  \text{and}\ R_{i_F,j_F(a)}\geq \widetilde{N}\geq N').
			\end{aligned}
		\end{equation}
		
		As $|F_{n}\setminus\cup \mathcal{F}|\leq (\delta + \delta^{\frac{1}{4}})|F_{n}|$ by the construction of applying
Lemma \ref{coveringlem}, one has

		\[
		\begin{aligned}
			N(\mathcal{U}^{F_{n}},W_{\mathcal{F}})&\leq |\mathcal{U}|^{|F_{n}\setminus\cup \mathcal{F}|} \cdot\prod_{F\in \mathcal{F}}
			N (\mathcal{U}^{F}, W_{\mathcal{F}})
			\\&\leq |\mathcal{U}|^{(\delta+\delta^{\frac{1}{4}})|F_{n}|}\cdot\prod_{F\in \mathcal{F}}\mathrm{e}^{\lambda_\delta|F|}\ \ \ (\text{using \eqref{dk}})
			\\&\leq |\mathcal{U}|^{(\delta+\delta^{\frac{1}{4}})|F_{n}|}\cdot\exp{\Bigg({\frac{\lambda_\delta|\cup \mathcal{F}|}{1-10\delta^{\frac{1}{4}}}}\Bigg)}\ \ \ (\mbox{as}\,\, \mathcal{F} \,\,\text{is}\,\,10\delta^{\frac{1}{4}} \text{-disjoint})
			\\&\leq  |\mathcal{U}|^{(\delta+\delta^{\frac{1}{4}})|F_{n}|}\cdot \mathrm{e}^{\lambda|F_{n}|}\ \ \ (\text{note}\ \lambda_\delta=(1-10\delta^{\frac{1}{4}})\lambda\ \text{and}\ \cup \mathcal{F}\subset F_n).
		\end{aligned}
		\]
		This finishes the proof of \eqref{estimate-number of disjoint}, and hence finishes the whole proof.
	\end{proof}
	
	Now we are ready to prove Theorem \ref{thm3}.
	
	\begin{proof}[Proof of Theorem \ref{thm3}]
		With the help of \cite[Theorem  13.3]{Dooley-Zhang2015} we already obtain that
\[h_{\mathrm{top}}(G,X|\pi)
		= \sup_{y\in Y} h_{\mathrm{top}}(\pi^{-1}(y),\{F_n\}_{n\in\mathbb N})
		\ge \sup_{y\in Y} h_{\mathrm{dim}}(\pi^{-1}(y),\{F_n\}_{n\in\mathbb N})
		\]
for any F\o lner sequence $\{F_n\}_{n\in\mathbb{N}}$, where the inequality follows from the basic fact that
\[h_{\mathrm{top}}(\pi^{-1}(y),\{F_n\}_{n\in\mathbb N})
		\ge h_{\mathrm{dim}}(\pi^{-1}(y),\{F_n\}_{n\in\mathbb N})\]		
(see for example \cite[Proposition 5.1]{Dou-Wang-Zhang2025}).

Therefore, it remains to prove that, under the temperedness assumption and (\ref{1.1}),
		\begin{equation} \label{section-5}
h_{\mathrm{top}}(G,X|\pi)
		\le \sup_{y\in Y} h_{\mathrm{dim}}(\pi^{-1}(y),\{F_n\}_{n\in\mathbb N}).
\end{equation}
		Suppose that $\sup_{y\in Y}h_{\mathrm{dim}}(\pi^{-1}(y),\{F_{n}\}_{n\in\mathbb N})<\infty$ otherwise there is nothing to prove.
		Take arbitrarily $\lambda\in \mathbb{R}$ satisfies $\lambda>\sup_{y\in Y}h_{\mathrm{dim}}(\pi^{-1}(y),\{F_{n}\}_{n\in\mathbb N})$ and
		put
		$
		\widetilde F_n:=F_n\cup\{e_G\}$ for each $n\in\mathbb N.
		$
		By Lemma~\ref{prop-folnerseq},
		\(\{\widetilde F_n\}_{n\in\mathbb N}\) is still a tempered F\o lner sequence
		satisfying \eqref{1.1}, and \(e_G\in\widetilde F_n\) for every \(n\), and
then, by Proposition~\ref{prop16}, for every
		\(\mathcal U\in\mathfrak C_X^o\),
		\[
		\sup_{y\in Y}
		h_{\mathrm{dim}}(\pi^{-1}(y),\mathcal U,
		\{\widetilde F_n\}_{n\in\mathbb N})
		<\lambda.
		\]
Now applying Proposition \ref{prop15} to the tempered F\o lner sequence \(\{\widetilde F_n\}_{n\in\mathbb N}\), one has that, for each  $\mathcal{U}\in \mathfrak{C}_{X}^{o}$ and every small $\delta>0$, there exists $P\in \mathbb{N}$ such that
		$$
		\sup_{y\in Y}N(\mathcal{U}^{\widetilde F_{n}},\pi^{-1}(y))\leq |\mathcal{U}|^{(\delta+\delta^{\frac{1}{4}})|\widetilde F_{n}|}\cdot \mathrm{e}^{\lambda|\widetilde F_{n}|}\cdot 2^{\delta|\widetilde F_{n}|}\quad \text{ for each $n>P$}.
		$$
		Thus, by taking logarithms, dividing by \(|\widetilde F_n|\), and then
		letting \(n\to\infty\), we obtain, using the independence of
		\(h_{\mathrm{top}}(G,\mathcal U|\pi)\) from the choice of the F\o lner sequence,
		that
		\[
		h_{\mathrm{top}}(G,\mathcal{U}|\pi)\leq \lambda+(\delta+\delta^{\frac{1}{4}})\log |\mathcal{U}|+\delta.
		\]
		Since \(\delta>0\) can be chosen arbitrarily small, letting \(\delta\to0\) we have
		$
		h_{\mathrm{top}}(G,\mathcal U|\pi)\leq \lambda
		$
		for every \(\mathcal U\in\mathfrak C_X^o\).
	Now letting
		$
		\lambda\searrow
		\sup_{y\in Y}h_{\mathrm{dim}}(\pi^{-1}(y),\{F_n\}_{n\in\mathbb N})
		$ and taking the supremum over \(\mathcal U\in\mathfrak C_X^o\),
		we obtain the required inequality \eqref{section-5}.
		Thus Theorem~\ref{thm3} follows.
	\end{proof}

	\section*{Acknowledgements}

	We would like to thank Dou Dou for sharing
with us his recent preprint \cite{Dou-Wang-Zhou2026-preprint}, and thank Ruifeng Zhang for useful discussions. Guohua Zhang is supported by the National Key Research and Development Program of China (No. 2021YFA1003204). This
work has been supported by the New Cornerstone Science Foundation through the New
Cornerstone Investigator Program.

	\bibliographystyle{alpha}
	
\bibliographystyle{amsplain}

\begin{thebibliography}{DWWZ26}

\bibitem[AKM65]{Adler-Konheim-McAndrew1965}
Roy~L. Adler, Alan~G. Konheim, and M.~H. McAndrew.
\newblock Topological entropy.
\newblock {\em Trans. Amer. Math. Soc.}, 114:309--319, 1965.

\bibitem[Bow72]{R.Bowen1972}
Rufus Bowen.
\newblock Entropy-expansive maps.
\newblock {\em Trans. Amer. Math. Soc.}, 164:323--331, 1972.

\bibitem[Bow73]{R.Bowen1973-TAMS}
Rufus Bowen.
\newblock Topological entropy for noncompact sets.
\newblock {\em Trans. Amer. Math. Soc.}, 184:125--136, 1973.

\bibitem[CZ15]{Chung-Zhang2015}
Nhan-Phu Chung and Guohua Zhang.
\newblock Weak expansiveness for actions of sofic groups.
\newblock {\em J. Funct. Anal.}, 268(11):3534--3565, 2015.

\bibitem[Dan01]{Danilenko2001}
Alexandre~I. Danilenko.
\newblock Entropy theory from the orbital point of view.
\newblock {\em Monatsh. Math.}, 134(2):121--141, 2001.

\bibitem[Dow11]{Downarowicz2011}
Tomasz Downarowicz.
\newblock {\em Entropy in dynamical systems}, volume~18 of {\em New
  Mathematical Monographs}.
\newblock Cambridge University Press, Cambridge, 2011.

\bibitem[DWWZ26]{Downarowicz-Weiss-Wiecek-Zhang2026-preprint}
Tomasz Downarowicz, Benjamin Weiss, Mateusz Wi\k{e}cek, and Guohua Zhang.
\newblock Universality of {$G$}-subshifts with specification.
\newblock {\em preprint}, 2026.

\bibitem[DWZ25]{Dou-Wang-Zhang2025}
Dou Dou, Ying Wang, and Guohua Zhang.
\newblock New characterizations of topological conditional entropy for actions
  of amenable groups.
\newblock {\em Fund. Math.}, 271(1):71--96, 2025.

\bibitem[DWZ26]{Dou-Wang-Zhou2026-preprint}
Dou Dou, Xiaochen Wang, and Xiaomin Zhou.
\newblock Shannon-{M}c{M}illan-{B}reiman type theorems for amenable groups.
\newblock {\em preprint}, 2026.

\bibitem[DZ15]{Dooley-Zhang2015}
Anthony~H. Dooley and Guohua Zhang.
\newblock Local entropy theory of a random dynamical system.
\newblock {\em Mem. Amer. Math. Soc.}, 233(1099):vi+106, 2015.

\bibitem[DZ22]{Downarowicz-Zhang2022-IJM}
Tomasz Downarowicz and Guohua Zhang.
\newblock Tail variational principle and asymptotic {$h$}-expansiveness for
  amenable group actions.
\newblock {\em Israel J. Math.}, 251(1):301--325, 2022.

\bibitem[DZ23]{Downarowicz-Zhang2023}
Tomasz Downarowicz and Guohua Zhang.
\newblock Symbolic extensions of amenable group actions and the comparison
  property.
\newblock {\em Mem. Amer. Math. Soc.}, 281(1390):vi+95, 2023.

\bibitem[DZZ23]{Dou-Zheng-Zhou2023}
Dou Dou, Dongmei Zheng, and Xiaomin Zhou.
\newblock Packing topological entropy for amenable group actions.
\newblock {\em Ergodic Theory Dynam. Systems}, 43(2):480--514, 2023.

\bibitem[FH12]{Feng-Huang2012}
De-Jun Feng and Wen Huang.
\newblock Variational principles for topological entropies of subsets.
\newblock {\em J. Funct. Anal.}, 263(8):2228--2254, 2012.

\bibitem[Gro99]{Gromov1999}
Misha Gromov.
\newblock Topological invariants of dynamical systems and spaces of holomorphic
  maps. {I}.
\newblock {\em Math. Phys. Anal. Geom.}, 2(4):323--415, 1999.

\bibitem[GTW00]{Glasner-Thouvenot-Weiss2000}
Eli Glasner, Jean-Paul Thouvenot, and Benjamin Weiss.
\newblock Entropy theory without a past.
\newblock {\em Ergodic Theory Dynam. Systems}, 20(5):1355--1370, 2000.

\bibitem[HYZ11]{Huang-Ye-Zhang2011}
Wen Huang, Xiangdong Ye, and Guohua Zhang.
\newblock Local entropy theory for a countable discrete amenable group action.
\newblock {\em J. Funct. Anal.}, 261(4):1028--1082, 2011.

\bibitem[HZZ26]{He-Zhang-Zhang2026-preprint}
Xinyao He, Guohua Zhang, and Ruifeng Zhang.
\newblock Metric mean dimension of amenable group actions: localization and
  non-uniformity.
\newblock {\em arXiv:2606.13270}, 2026.

\bibitem[Kat07]{Katok2007}
Anatole Katok.
\newblock Fifty years of entropy in dynamics: 1958--2007.
\newblock {\em J. Mod. Dyn.}, 1(4):545--596, 2007.

\bibitem[Kol58]{Kolmogorov1958}
A.~N. Kolmogorov.
\newblock A new metric invariant of transient dynamical systems and
  automorphisms in {L}ebesgue spaces.
\newblock {\em Dokl. Akad. Nauk SSSR (N.S.)}, 119:861--864, 1958.

\bibitem[Lin01]{Lindenstrauss2001}
Elon Lindenstrauss.
\newblock Pointwise theorems for amenable groups.
\newblock {\em Invent. Math.}, 146(2):259--295, 2001.

\bibitem[Mis76]{Misiurewicz1976-SM}
Micha{\l} Misiurewicz.
\newblock Topological conditional entropy.
\newblock {\em Studia Math.}, 55(2):175--200, 1976.

\bibitem[OW87]{Ornstein-Weiss1987}
Donald~S. Ornstein and Benjamin Weiss.
\newblock Entropy and isomorphism theorems for actions of amenable groups.
\newblock {\em J. Analyse Math.}, 48:1--141, 1987.

\bibitem[OZ11]{Oprocha-Zhang2011}
Piotr Oprocha and Guohua Zhang.
\newblock Dimensional entropy over sets and fibres.
\newblock {\em Nonlinearity}, 24(8):2325--2346, 2011.

\bibitem[Pes97]{Pesin1997}
Yakov~B. Pesin.
\newblock {\em Dimension theory in dynamical systems}.
\newblock Chicago Lectures in Mathematics. University of Chicago Press,
  Chicago, IL, 1997.
\newblock Contemporary views and applications.

\bibitem[WZ25]{Wang-Zhang2025-preprint}
Xulei Wang and Guohua Zhang.
\newblock Smooth surface systems may contain smooth curves which have no
  measure of maximal entropy.
\newblock {\em arXiv:2505.10458}, 2025.

\bibitem[Yan15]{Yan2015}
Kesong Yan.
\newblock Conditional entropy and fiber entropy for amenable group actions.
\newblock {\em J. Differential Equations}, 259(7):3004--3031, 2015.

\bibitem[YZ07]{Ye-Zhang2007}
Xiangdong Ye and Guohua Zhang.
\newblock Entropy points and applications.
\newblock {\em Trans. Amer. Math. Soc.}, 359(12):6167--6186, 2007.

\bibitem[ZC16]{Zheng-Chen2016}
Dongmei Zheng and Ercai Chen.
\newblock Bowen entropy for actions of amenable groups.
\newblock {\em Israel J. Math.}, 212(2):895--911, 2016.

\bibitem[ZZC15]{Zhou-Zhang-Chen2015}
Xiaoyao Zhou, Yaqing Zhang, and Ercai Chen.
\newblock Topological conditional entropy for amenable group actions.
\newblock {\em Proc. Amer. Math. Soc.}, 143(1):141--150, 2015.

\end{thebibliography}

\end{document}